\documentclass[11pt,english]{article}
\usepackage[T1]{fontenc}
\usepackage[utf8]{luainputenc}
\usepackage{xcolor}
\usepackage{pdfcolmk}
\usepackage{amsthm}
\usepackage{amsmath}
\usepackage{stmaryrd}
\usepackage{stackrel}
\usepackage{setspace}
\PassOptionsToPackage{normalem}{ulem}
\usepackage{ulem}

\makeatletter

\providecolor{lyxadded}{rgb}{0,0,1}
\providecolor{lyxdeleted}{rgb}{1,0,0}

\numberwithin{figure}{section}
\numberwithin{equation}{section}
\numberwithin{table}{section}

  \input amssymb.sty
\usepackage{etoolbox}
\patchcmd{\thebibliography}{\section*}{\section}{}{}
\renewcommand{\refname}{REFERENCES}
\usepackage{algpseudocode} \usepackage{algorithm}
\usepackage{algpseudocode,algorithm}
\usepackage{lipsum}
\usepackage{caption}
\usepackage{graphics}

\usepackage{amsmath}
\usepackage{amssymb}
\usepackage[all]{xy}

\usepackage{epsfig}\usepackage{youngtab}

\newcommand{\ef}{\end{equation}}
\chardef\bslash=`\\ 

\newdir{:=}{{}}
\newcommand*\colvec[3][]{
    \begin{pmatrix}\ifx\relax#1\relax\else#1\\\fi#2\\#3\end{pmatrix}
}

\newtheorem*{thm*}{Theorem}

\newtheorem{lem}{Lemma}[section]
\newtheorem*{lem*}{Lemma}

\newtheorem*{corl*}{Corollary}

\newtheorem{prop}{Proposition}[section]
\newtheorem{prop*}{Proposition}

\theoremstyle{definition}
\newtheorem{defn}{Definition}[section]
\newtheorem{examp}{Example}
\newtheorem*{examp*}{Example}
\newtheorem*{remark*}{Remark}
\newtheorem*{CC*}{Crossover Conjecture}
\newtheorem*{Note*}{Note}
\newtheorem*{defn*}{Definition}

 \theoremstyle{remark}

\renewcommand{\tt}{\tilde{t}}
\newcommand{\ovr}{\overline}

 \renewcommand{\sectionmark}[1]{}

 \newcommand{\id}{\operatorname{id}}

\newcommand{\Rad}{\operatorname{Rad}}
\newcommand{\End}{\operatorname{End}}
\newcommand{\Hom}{\operatorname{Hom}}

\makeatother

\usepackage{babel}
\begin{document}

\title{Elementary equivalences for blocks with normal elementary abelian defect group of rank 2} 
\author{Mary Schaps and Zehavit Zvi}

\maketitle

 \begin{abstract}

We consider the effect of performing an elementary equivalence as defined by Okuyama on a group block of form $F(C_p \times C_p)\rtimes C_r$, for a field $F$ of characteristic $p$. If $I=\{0,1,2,\dots r-1\}$ is the set of residues corresponding to the simple modules of $FC_r$, the elementary equivalence is determined by a proper, non-empty subset $I_0 \subset I$, and the corresponding elementary tilting complex is completely determined by $I_0$.

We give a catalog of homogeneous maps between irreducible components of the elementary tilting complex.  When the subset $I_0$ is an interval, we prove that the maps in the catalog are sufficient to describe all homogeneous maps between two irreducible components.

\end{abstract}
\section{INTRODUCTION}

\noindent This work concerns blocks over a field $F$ of odd finite characteristic $p$, with elementary abelian defect group $D=C_p \times C_p$. We restrict to odd characteristic because the case of $C_2 \times C_2$ has been intensively studied.  The defect group $C_p \times C_p$  is the defect group with the next largest rank after the well-studied case of cyclic defect group.

We are interested in the tilting theory
of these blocks, trying to apply some of the insights gathered from the cyclic defect case \cite{R1,SZ2,SZ}. 
In that case, every block can be obtained from a block with normal defect group by a sequence of elementary tiltings known as mutations.  Okuyama \cite{O} found some important group blocks as the result of a sequence of more general elementary equivalences in the case $p=3$. The Brou\'e conjecture for principal blocks when $p=3$ was eventually proven by Koshitani and Kunugi in \cite{KK} but is still open for larger primes. Unfortunately, one of Okuyama's central lemmas is less useful for primes $p \geq 5$, so his method had not found ready application to other primes.

A study of the blocks of  subgroups of the groups in the Atlas of Simple Groups \cite{C} shows that most of the abelian defect groups are either cyclic or of the form $D=C_p \times C_p$, and all have a property we will define later as being balanced. Furthermore, the Brauer correspondents of these blocks are usually of the form $D \rtimes C_r$ or $D \rtimes D_n$. In this work, we are taking up the first case, $D \rtimes C_r$ and asking what sort of algebras can be reached by a single elementary equivalence, as a function of the prime $p$ and the number $r$ of simple modules in the original block of cyclic defect.  

\section{DEFINITIONS AND NOTATION}

\subsection{Blocks with normal defect groups}
\label{sec: Blocks with normal defect group}

The blocks with normal defect groups were studied in depth by several authors \cite{Lan}, \cite{RR}, \cite{Sw}  and also \cite{K}. We will review that part of the theory which we need. By Brauer's First Main Theorem, we know that there is a one-to-one correspondence between the blocks $B$ of $FG$ with defect group $D$, and the blocks $b$ of $FN_G(D)$ with defect group $D$ \cite{CR}. In order to investigate the relationship between $B$ and $b$ we must know more about the structure of the blocks of $FN_G(D)$. We restricted ourselves to Brauer correspondents which are Morita equivalent as blocks to $F[(C_p \times C_p)\rtimes H]$, with $H=C_r$ for $r$ prime to $p$. We will show that we can choose a basis for the radical of $F[D]$ which are eigenvectors for  the generator of $C_r$ and we will assume that the eigenvalues of the generators of this radical are inverse to each other.

\subsection{The quiver of a block}
\label{sec:The quiver of a local block}
Let $A$ be a $K$-algebra, where $K$ is an algebraically closed field with $A/\Rad(A) \xrightarrow{\sim}\bigoplus M_{n_{i}}(K)$. Let $f_1, \dots, f_t$ be a set of primitive idempotents, one from each matrix block, and set 
\[
n'_{ij}=\dim (f_i(J/J^2)f_j)
\]
for $J=\Rad(A)$.
\begin{defn} The \textbf{quiver} of $A$ is the directed graph with $t$ vertices, where $t$ is the number of simple modules, and $n'_{ij}$ arrows from $j$ to $i$. Then $n'_{ij}$ is the number of irreducible morphisms from $j$ to $i$.
\end{defn}

In \cite{S}, results were quoted showing that:
\begin{enumerate}
	\item The block idempotents of $FH$ embed in $FG$.
	\item $\Rad(FG)$ is generated by $J=\Rad(FD)$.
	\item The action of $H$ on $D$ induces an action of $H$ on $J/J^2$ with character $\chi$, and there is an $H$-submodule $N$ of $J$ isomorphic to $J/J^2$. 
	\item The quiver of $G$ is the McKay graph $D(H, \chi)$ for the character $\chi$ of $H$.
\end{enumerate}

Gerstenhaber and Schaps go on to show in \cite{GS} that the $H$-module $J/J^2$ can be embedded as a submodule of $J$ of dimension $d$ over $F$. If $H$ is abelian, it decomposes as a direct sum of one-dimensional submodules, so that we can find a basis for $J$ of idempotents for the action of $H$.

\subsection{The McKay graph}
\label{sec:The McKay graph}
In \cite{M}, McKay introduced a graph that he called the representation graph and which later was called the McKay graph by Schaps, Shapira and Shlomo in \cite{SSS}.

Let $H$ be a finite group, and let $F$ be a field of characteristic zero or of positive characteristic $p$,  $ p \nmid \left|H\right|$. Let $K$ be an extension of $F$ which is a splitting field for $H$. Let the character table of $H$ be considered as having entries in $K$. Let $\chi$ be a character of $H$, with kernel $N$. If $\theta_1,\dots,\theta_t$ is the complete set of irreducible characters of $H$, and $\chi \cdot\theta_j$ is defined by pointwise multiplication of the values for the various conjugacy classes, define the natural number $n_{ij}$ as the multiplicity of the constituent $\theta_i$ in the product character $\chi \cdot \theta_j$. Since the $\theta_i$ form an orthonormal basis, with respect to an inner product ( -, -), this is given by $n_{ij}=(\theta_i,\chi\cdot\theta_j$).
\begin{defn} If $H$ is a finite group, $K$ is a sufficiently large field of characteristic prime to $\left|H\right|$, and $\chi$ is a $K$ character of $H$, the \textbf{McKay graph} $D$($H$,$\chi$) is the directed graph with
	\begin{enumerate}
		\item[(i)] $t$ vertices labelled by the irreducible characters $\theta_1,\dots,\theta_t$ of $H$.
		\item[(ii)] For each pair $(\theta_i$,$\theta_j)$ of characters, $n_{ij}$ arrows from $\theta_j$ to $\theta_i$, where 
		\[
		n_{ij}=(\theta_i,\chi \cdot\theta_j).
		\]
	\end{enumerate}
\end{defn}
In (\cite{SSS}, \S 3), it was shown that for a graph algebra of the form  $(C_p^d) \rtimes H$, where $ p \nmid \left|H\right|$, the quiver was given by the McKay graph, and a method was also given in \cite{SSS} for writing down the relations on the Morita equivalent basic algebra. 

\subsection{The group algebra $F[(C_p\times C_p) \rtimes C_r]$}
\label{The group algebra}

We now restrict to the case most common for non-cyclic defect, the case of $d=2$ and $H=C_r$, with $r$ not divisible by $p$.  By the results we quoted from \cite{GS}, we can find two generators $x$ and $y$ for the radical $J$ which are eigenvectors for the action of $C_r$.
There are many ways to choose a pair of eigenvalues of the action of $C_r$, but what usually occurs in the examples is that there is an automorphism that interchanges $x$ with $y$, and the generator of the $C_r$ with its inverse, which means that the idempotent of $y$ is the inverse of the idempotent of $x$, so that the action of $C_r$ on $xy$ is trivial.

\begin{defn} Let $G=(C_p \times C_p) \rtimes C_r$, and let $x$ and $y$ be generators of $J$ which are eigenvectors for the action of $C_r$. We will say that the action is \textit{balanced} if the eigenvalue of the action of $C_r$ on $x$ is the inverse of the eigenvalue of $y$.
	\end{defn}

\begin{examp}
	\label{example 554_2}
	If $H$ is $C_4=\langle a \rangle$,  $G=(C_5 \times C_5) \rtimes C_4$, where the generator $a$ of $C_4$ acts on the generators $c$ and $d$ of $C_5 \times C_5$ by sending $c$ to $c^2$ and $d$ to $d^3$. Let $i \in F$ be a primitive 4th root of unity, with irreducible characters  $\theta_j(a^k)= i^{(j-1)k}$, where $j=1, \dots, 4$ and $k=0, \dots, 3$. Let $e_u$, $u=1,\dots, 4$, be the idempotents of $KC_4$. Set $\chi=\theta_2+ \theta_4$.  
	Then, 
	\[
	\chi \cdot\theta_j= \theta_{j+1}+\theta_{j-1}
	\]
	where the indices are taken mod 4. Since $2$ and $3$ are primitive fourth roots of unity modulo $5$,
we can choose the generators as follows
\begin{align*}
	x &= c+3c^2+2c^3+4c^4 \\
	y &= d+2d^3+3d^2+4d^4. \\
\end{align*}
\noindent Then $a^{-1}xa=c^2+3c^4+2c+4c^3=2x$ and 
$a^{-1}ya=d^3+2d^4+3d+4d^2=3y$
	
	The quiver of $FG$ is given by
	\[
	\quad
	
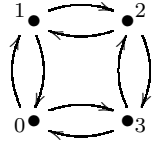
\xymatrix{
		{^1 \bullet} \ar@/^/[r] \ar@/^/[d] & {\bullet^2} \ar@/^/[l] \ar@/^/[d]\\
		{_0 \bullet} \ar@/^/[r] \ar@/^/[u] & {\bullet_3} \ar@/^/[l] \ar@/^/[u]
	}
	\]
	with the relations:
	
	\begin{align*}
		x_{u(u+1)}&=e_uxe_{u+1}\\
		y_{u(u-1)}&=e_uye_{u-1}\\
		x&=x_{01}+x_{12}+x_{23}+x_{30}\\
		y&=y_{10}+x_{03}+x_{32}+x_{21}\\
		xy&=yx\\
		x^5&=0, y^5=0\\
	\end{align*}
	
\end{examp}

\subsection{The structure of the projective indecomposable modules $P_t$ for $F[(C_p \times C_p) \rtimes C_r]$}
\label{The structure of projective modules}
Since $C_r$ is abelian, the simple modules of $FG$, which correspond to the simple modules of $FC_r$, are one-dimensional. They bijectively correspond to the idempotents of $FC_r$. Using the fact that powers of $\xi$ and $a$ can be taken modulo $r$, we recall that the central idempotents of $FC_r$ are given by
\begin{align*}
	e_t&= \frac{1}{r} \sum_{i=0}^{r-1} \xi^{-it} a^i, t=0, \dots, r-1.\\
\end{align*}
We consider one particular idempotent $e_t$ and its projective module $P_t$. We first consider the vector space $S_t=\langle e_t \rangle $, and show that it is a one-dimensional $C_r$-module under multiplication in $FC_r$.
\begin{align*}
	a \cdot e_t&=\frac{1}{r} \sum_{i=0}^{r-1} \xi^{-it} a^{i+1} \\
	&= \xi^{t} \cdot \frac{1}{r} \sum_{i=0}^{r-1} \xi^{-(i+1)t} a^{i+1}\\
	&= \xi^{t} e_t.\\
\end{align*}
By \S2.3, for $A=FG$, the projective modules are $P_t=FGe_t$. If $J=\Rad(FD)$ and $I=\Rad(FG)=JFG$ then $\Rad(P_t)=I \cap P_t$ and $S_t=P_t/ \Rad(P_t)$ is the simple module corresponding to $e_t$.

We have assumed that the characteristic of $F$ is $p$. Since $D  \xrightarrow{\sim}C_p \times C_p$, $FD= F[x, y]/(x^p, y^p)$. Since $C_r$ is an abelian group and $p \nmid r$ we set $FC_r= \oplus Fe_t$. The structure of these blocks is treated in \cite{Lan}, \cite{Sw}, \cite{K}, \cite{RR}. Our notation follows that of \cite{GS} in the simpler case where there is no inertial subgroup. We know that the dimension of $FD$ is $p^2$ and there is a basis $x^iy^j$. The module can be represented as a diamond of length $p$ on each side.

\subsection{Derived equivalence}

\noindent We fix an abelian category $\mathcal C$ and we denote by  $Ch(\mathcal C)$
the category of cochain complexes of objects of $\mathcal C$. 
The differentials of the complex are 
morphisms $\left\{ d_{n}:C_{n}\rightarrow C_{n+1}\right\} $
satisfying  $d_{n+1} \circ d_{n}=0$. For any given complex $C^\bullet$, the 
cocycles  $Z^{n}(C^\bullet)$, coboundaries  $B^{n}(C^\bullet)$, and  cohomology modules 
$H^{n}(C^\bullet)=Z^{n}(C^\bullet)/B^{n}(C^\bullet)$ are defined as in such standard texts as \cite{W}.

~
\noindent \begin{defn}

\noindent A morphism $f_\bullet:C^\bullet\rightarrow D^\bullet$ of cochain complexes is called
a \textit{quasi-isomorphism }if the induced maps $H^{n}\left(C^\bullet\right)\rightarrow H^{n}(D^\bullet)$
are all isomorphisms.

\noindent \end{defn}

\noindent \begin{defn}

\noindent A cochain complex $C$\textbf{ }is called \textit{bounded}\textbf{
}if almost all the $C^{n}$ are zero. The\textit{ }complex $C$\textbf{ }is \textit{bounded
below}\textbf{ }if there is a bound $a$\textbf{ }such that $C^{n}=0$
for all $n<a$. The cochain complexes which are partially or fully bounded
form full subcategories $Ch^{b},\thinspace Ch^{+},\thinspace Ch^{-}$of
$Ch(\mathcal C)$.

\noindent \end{defn}

\noindent \begin{defn}

\noindent The \textit{derived category} of an abelian category $\mathcal C$ is the category obtained from  $Ch(\mathcal C)$ by adding
formal inverses to all the quasi-isomorphisms between chain complexes.  It is called the \textit{bounded derived category} and denoted  $D^{b}(\mathcal C)$ if we consider only bounded complexes.
A \textit{derived equivalence} between two abelian categories is an equivalence of categories between their derived categories.
\noindent \end{defn}

\noindent \begin{defn}
\noindent We say that a chain map $f:C\rightarrow D$ is \textit{null
homotopic} if there are maps $s_{n}:C^{n}\rightarrow D^{n+1}$ such
that $f=ds+sd$.

\noindent \end{defn}

\noindent \begin{defn}

\noindent Two chain maps $f,g:C\rightarrow D$ are \textit{chain homotopic}
if their difference $f-g$ is null homotopic, in other words, if $f-g=sd+ds$
for some $s$. The maps $\left\{ s_{n}\right\} $ are called a \textit{homotopy}
 from $f$ to $g$. We say that $f:C\rightarrow D$ is a
\textit{homotopy equivalence} if there is a map $g:D\rightarrow C$
such that $g\circ f$ is chain homotopic to the identity map $id_{C}$
and $f\circ g$ is chain homotopic to the identity map $id_{D}$.
\end{defn}

\subsection{Tilting complexes}

\noindent We now consider complexes $T$ of projective left modules over an associative ring $R$. The
notation $T[n]$ denotes the complex that is isomorphic to $T$ as
a module but in which the gradation has been shifted $n$ places to
the left and the differential is the shift of the differential multiplied
by $(-1)^{n}$.

~

\noindent For any ring $R$, let $D^{b}(R)$ be the derived category
of bounded complexes of $R$-modules.

\noindent \begin{defn}\label{TC}Let $R$ be a Noetherian ring. A bounded complex
$T$ of finitely generated projective $R$-modules is called a \textit{tilting
complex} if:

\noindent \renewcommand{\labelenumi}{(\roman{enumi})}
\begin{enumerate}
\item \noindent Hom\textsf{$_{D^{b}\left(R\right)}\left(T,T[n]\right)=0$}
whenever\textsf{ $n\neq0$}.
\item \noindent For any indecomposable projective\textsf{ $P$}, define
the stalk complex to be the complex\textsf{ $P^{.}:0\rightarrow P\rightarrow0$.}
Then every such\textsf{ $P^{.}$ }is in the triangulated category
generated by the direct summands of direct sums of copies of\textsf{
$T.$}
\end{enumerate}
 \end{defn}

\noindent A complex $T$ satisfying only (i) is called
a \textit{partial tilting complex.}

\begin{singlespace}
\noindent \begin{defn}Fix an abelian category $\mathcal C$ and the
category of cochain  complexes $Ch(\mathcal{C})$. For two complexes
$X$\textbf{ }and $Y$\textbf{ }denote by $Z(X,Y)$ the set of morphisms
from $X$\textit{ }to $Y$\textbf{ }which are homotopic to zero. The
collection of all $Z(X,Y)$\textbf{ }forms a subgroup of Hom$_{C(\mathcal{C})}(X,Y)$.
Denote by $K(\mathcal{C})$\textbf{ }the\textbf{ }quotient category,
i.e.{\large{} }$K(\mathcal{C})$ is the category having the same objects
as $Ch(\mathcal{C})$ but with morphisms
 \begin{center} 
{Hom${{}_{K}}_{(\mathcal{C}}{_{)}(X,Y)=}$ Hom${{}_{Ch}}$$_{(\mathcal{C}}$${_{)}(X,Y)/Z(X,Y)}$,}
\end{center}
so that two homotopic maps are identified.
The quotient category $K(\mathcal{C})$\textbf{ }is called the  \textit{homotopy
category},  and a homotopy equivalence between 
complexes is an isomorphism in the homotopy category. For an abelian category ${\mathcal{C},}$ $K^{-}(\mathcal{C})$
is the homotopy category of right bounded complexes in ${\mathcal{C}}$,
and similarly one can define $K^{+}(\mathcal{C})$.\end{defn}
\end{singlespace}

~

\noindent The derived category is not an abelian category, but it is a triangulated category. 
The original theory of tilting concerned modules called tilting modules. 
Happel, in \cite{H} showed that if there was a tilting between two algebras  $\Lambda$ and $\Gamma$,
it induced a functor which was an equivalence of their derived bounded categories.
Rickard \cite{R1} then proved a converse when tilting modules were replaced by tilting complexes,
namely, that  there is a tilting complex $T$ over $\Lambda$ with endomorphism
ring End$_{D^{b}\left(\Lambda\right)}(T)^{op}\cong\Gamma$.

~

\noindent Let
$\Lambda-Proj$ be the abelian category of projective $\Lambda$-modules.
Let $Sum-T$ be the additive category of all the direct sums of copies
of $T$ (as a full subcategory of $K^{-}(\Lambda-Proj))$.
From Lemma 3.3.1 of \cite{KZ2}  we have that $K^{-}(\Gamma-Proj)$
is equivalent to $K^{-}\left(Sum-T\right)$ as triangulated categories.
Thus, every tilting complex can be transformed into a projective complex, 
and to calculate the required endomorphism ring of the 
tilting complex, it will be sufficient to find chain maps between the 
components of the tilting complex which are not homotopic to zero.

\subsection{Elementary equivalences}
\label{sec:ELEMENTARY EQUIVALENCES}
Of particular importance in this work are the elementary equivalences, which go back to the work of Rickard \cite{R2} and
Okuyama \cite{O}. 
\noindent \begin{defn}\label{Elementary}
	Let $A$ be a finite-dimensional basic algebra, with projective modules $P_j, j \in I$. 

Fix a subset   $I_0 \subset I$.
For any $i\in I$ we define a complex by 

$T_{i}=\begin{cases}
\begin{array}{cccc}
(0th) & \, & (1st) & \,\\
P_{i} & \longrightarrow & 0 & \:\:i \in I_0\\
Q_{i} & \overset{\pi_{i}}{\longrightarrow} & P_{i} & \:\:i \in I-I_0
\end{array} & \,\end{cases}$

\noindent where $Q_{i}\overset{\pi_{i}}{\longrightarrow}P_{i}$ is
a minimal projective presentation of the minimal kernel with quotient having only simples from $I-I_0$.
Now we define $T:=\varoplus_{i\ensuremath{{ \in I}}}T_{i}$.
The algebra \textit{ elementary eqivalent  } of $A$ through $T$ is is $A'\cong (\End_{D^{b}(A)}T)^{op}$.
\end{defn}

\subsection{Homogeneous maps}

To each $i \in I$,
we can associate an idempotent $\tilde{e}_{i}$ with $1_{A}$=$\stackrel[i=1]{e}{\sum}$$\tilde{e}_{i}$.

\begin{defn}
 \label{sec:Definition 3.7}
For $i,j \in\{0, \dots, r-1\}$ let
$P_i=Ae_i
P_j=Ae_j$.
The map 
$ f:P_i \to P_j$
 is \textbf{homogeneous} if 
$ f(e_i)=\epsilon x^{k}y^{\ell}e_j$
for integers  $0\leq k,\ell\leq p-1$, $\epsilon \in F$.
\end{defn}
\begin{defn}
 \label{sec:Definition 3.8} 
Let 
\begin{align*}
 f:& \bigoplus_{i= 1}^r P_i \to  \bigoplus_{j= 1}^r P_k.\\ 
\end{align*}
Then $f$ will be called  \textbf{homogeneous} if every component in $\bigoplus_{ i,k= 1}^rHom( P_i, P_k)$ is a homogeneous map
$f_{ik}:P_i \to P_k$
as in Definition \ref{sec:Definition 3.7}.
\end{defn}
\begin{defn} 
Let $l^{\bullet}$ be the chain map:
\[\xymatrix@C=70pt@R=50pt{\bigoplus_{i= 1}^m P_{\mu_i} \ar@{->}[r]^{\{d_{it}\}} \ar@<-0.2ex>@{->}[d]_-{\{b_{ij}\}} & \bigoplus_{t= 1}^k P_{\kappa_t} \ar@<0.2ex>@{->}[d]^-{\{f_{tu}\}} \\   \bigoplus_{j= 1}^n P_{\nu_j} \ar@{->}[r]^{\{c_{ju}\}} &  \bigoplus_{u= 1}^{\ell} P_{\lambda_u}} \]
We say that $l^{\bullet}$ is \textbf{homogeneous} if every vertical and horizontal map is homogeneous as in Definition \ref{sec:Definition 3.8} , and for every pair $(i,u)$ and for all $j$ and $t$ we get 
\begin{align*}
\sum^{k}_{t=1} f_{tu} \circ d_{it} = \sum^{n}_{j=1} c_{ju} \circ b_{ij} \\
\end{align*}

\end{defn}
\begin{defn} 
A tilting complex  $T=\bigoplus_{ i= 1}^r T_i$ over $A$ is called homogeneous if each component of the differentials is $\pm1$ times a monomial in $x,y$ times the appropriate idempotent. 
\end{defn} 

\section{MAPS BETWEEN COMPONENTS OF AN ELEMENTARY TILTING COMPLEX}
\label{sec:catalog of maps}

In order to describe the result of acting on the group algebra $A=F[C_p \times C_p)) \rtimes C_r]$ by an elementary equivalence, we must calculate the endomorphism ring of the tilting complex  $T=\bigoplus_{ i= 1}^r T_i$. To simplify all calculations, we will assume that the tilting complex $T$ is homogeneous, a condition which can be achieved using Lemma \ref{homogeneous map lem} below. 

 Let $I = \{0,1,2 \dots , r-1\}$ be arranged in a circle in the clockwise direction. Since $I_0$  is non-empty, removing the points of $I_0$ separates the circle into a finite number of intervals $J_1, \dots ,J_{m}$ which we will call \textit{arcs}. For a particular $J_i$, where $1\leq i \leq m$, we let $u_i$ be the first element of $I_0$ in the counter-clockwise direction and let $v_i$ be the first element of $I_0$ in the clockwise direction. If $I_0$ contains a single point $w$, we will have $m=1$ and $u_1=v_1=w$. 

We denote the $r$-residue of $n \in \mathbb{Z}$ by $\bar{n}$. Any $t \in I-{I_0}$ belongs to an arc $J_i=
[\overline{u_i+1}, \dots, \overline{t-1}, t, \overline{t+1}, \dots, \overline{v_i-1}]$ with the understanding that $0$ follows $r-1$. The arc $J_i$ is preceded by an interval $I_0^i=[\ovr{v_{i-1}},\ovr{v_{i-1}+1},\dots,\ovr{u_i-1},\ovr{u_i}]\subseteq I_0$ and followed by the interval $I_0^{i+1}$.

\begin{defn}
	\label{sec:Definition in} 
	If $t,\tilde{t} \in I$ both lie in one of the intervals $I_0^i$,
	$I_0^{i+1}$ or $J_i$, we will write $t \preceq \tilde{t}$ to indicate  $\tilde{t}$ either equals $t$ or is subsequent to $t$ in the interval.
\end{defn}

\begin{lem}\label{short} If $t,\tilde{t} \in I$ both lie in $J_i$ and $s$ lies in $I_0^i$, then
	\[
	\overline{\tilde{t}-t}+ \overline{t-s}=\overline{\tilde{t}-s}.
	\]
	Similarly, if $s \in I_0^{i+1}$, then
	\[
	\overline{\tilde{t}-t}+ \overline{s-\tilde{t}}=\overline{s-t}.
	\]	 
\end{lem}
\begin{proof}
	Using the notation above, $I_0^i \cup J_i$ can be written as an ordered interval of length less than or equal to $r$. If $v=v_{i-1}$ is the first element of $I^i_0$, then we renumber $I$ by replacing each 
	$k \in I$ by $k'=\overline{k-v}$.  If we identify $s',u',\tilde{t}'$, which are residues, with the corresponding integers between $0$ and $r-1$, then we have $s' < t' \leq \tilde{t}'$, so that
	$t'-s'=\overline{t'-s'}$,$\tilde{t}'-t'=\overline{\tilde{t}'-t'}$ and $\tilde{t}'-s'=\overline{\tilde{t}'-s'}$. Since the renumbering is just a cyclic permutation and does not affect the residue of the differences between two residues, and since, as integers,
	\[
	(\tilde{t}'-t')                                 +(t'-s')=\tilde{t}'-s',
	\]
	we obtain the desired result.
\end{proof}

In this article, we fix one arc $J_i$ and  we determine all the irreducible maps  
internal to the sets $I_0^i \cup J_i \cup I_0^{i+1}$. If $m=1$, there is just a single arc $J_1$, and a single complementary interval $I_0$. Then the maps we will construct will generate the entire endomorphism ring. We will leave to a sequel the case where $m>1$ . For simplicity of notations, we will write $J=J_i$, $u=u_i, v=v_i$.

We first determine the maps between the projective modules in the block of normal defect group.

\begin{lem}
	\label{homogeneous map lem}
	We have a group block $F[(C_p\times C_p)\rtimes C_r]$ with indecomposible projective modules $P_0, P_1, \dots P_{r-1}$. Suppose we are given a homogeneous map $P_i \to P_k$, for $0\leq i,k\leq r-1$. 
	\begin{enumerate}
		\item
		If $\overline{k-i}<p$ then there are homogeneous maps $f(e_i)=x^{\overline{k-i}+hr}(xy)^{q}$, where 
		$h$ is an integer satisfying $0 \leq h \leq \frac{p-1-\overline{k-i}}{r}$ and $0\leq q \leq p-1-(\overline{k-i}+hr)$.	
		
		If $\overline{i-k}<p$ then there are homogeneous maps $f(e_i)=y^{\overline{i-k}+hr}(xy)^{q}$, where 
		$h$ is an integer satisfying $0 \leq h \leq \frac{p-1-\overline{i-k}}{r}$ and $0\leq q \leq p-1-(\overline{i-k}+hr)$.
		\item
		The homogeneous maps give a basis for the homomorphisms from $P_i \to P_k$ thus we get
		
		\[
		f(e_i)=1_{\overline{k-i}<p} \sum_{h=0}^{\lfloor  \frac{p-1-\overline{k-i}}{r}\rfloor} \sum_{q=0}^{\lfloor p-1-(\overline{k-i}+hr) \rfloor}  a_q^h \cdot x^{\overline{k-i}+hr}(xy)^{q}e_k
		\]
		
		\[
		+1_{\overline{i-k}<p} \sum_{h=0}^{\lfloor  \frac{p-1-\overline{i-k}}{r}\rfloor} \sum_{q=0}^{\lfloor p-1-(\overline{i-k}+hr) \rfloor}  b_q^h \cdot y^{\overline{i-k}+hr}(xy)^{q}e_k.
		\]
		\noindent Note that if $r>p$ then $h=0$.
	\end{enumerate}
\end{lem}
\begin{proof}
	\begin{enumerate}
		\item
		By standard arguments of representation theory, the image of $e_i$ in $P_k$ must be a linear combination of basis elements $e_ix^sy^te_k$. Starting with $e_k$, if we multiply by $x$ we get a basis element $e_{k-1}xe_k$. If we repeat this $s$ times we get $e_{k-s}x^se_{k}$. The first time that $e_{k-s}=e_i$ will be when $s=\overline{k-i}$. This can only occur if $\overline{k-i}<p$; otherwise we get $0$. If we continue multiplying $h$ times by $x^r$ we keep on returning to $e_i$ as long as $s=\overline{k-i}+hr<p$. Multiplication by $y$ is dual, giving the elements $e_{k+t}y^te_k$. For the first time that $e_{k+t}=e_i$ is when $t=\overline{i-k}$ and this can only occur if $\overline{i-k}<p$. If we continue multiplying $h$ times by $y^r$ we keep on returning to $e_i$ as long as $t=\overline{i-k}+hr<p$.
		
		If $e_ix^se_k \neq 0$ then also $e_ix^s (xy)^qe_k\neq 0$ as long as $s+q \leq p-1$. Dually if $e_iy^te_k \neq 0$ then also $e_iy^t (xy)^qe_k\neq 0$ as long as $t+q \leq p-1$. Note that if $i=k$ there are maps $f(e_i)= \sum_{q=0}^{p-1}  c_q (xy)^{q}e_k$.
		
		\item
		The general homomorphism is a linear combination of the homogeneous maps we found in $1$. We added the two characteristic functions  $1_{\overline{k-i}<p}, 1_{\overline{i-k}<p}$ to indicate when part of the sum is zero.
		
	\end{enumerate}
\end{proof}

We recall from Definition \ref{Elementary} that the irreducible components $T_i$ are of two different types dependent on $i$.  If $i \in I_0$, then $T_i$ is a stalk complex, whereas if $i \in I-I_0$, it is a complex with nonzero components in two degrees. We will be working in the homotopy category of complexes of projective modules and we will denote $T_i$ by $P'_i$ when regarded as an element of the homotopy category. What we will do is propose a catalog of irreducible maps between the various $P'_i$, and then demonstrate that a general homogenous map from $P'_i$ to $P'_j$ can be written as a composition of maps in the catalog.

\subsection{Catalog of intra-arc maps}\label{catalog of intra-arc maps}

We now list the irreducible maps from $ P'_t$ to $P'_{t'}$ for $t,t'$ in the interval 
$I_0^i \cup J_i \cup I_0^{i+1}$. The maps will generally be described by powers of $x$ and $y$, but we will also need the following special maps in the catalog:

\begin{defn} \label{Def. 42}
Let $t \in J$ in cyclic order, then the chain maps $(\xi_1,x)$, $(\xi_2,y)$ are as follows
\[
\begin{array}{cc}
\xymatrix@C=70pt@R=30pt{
P_u \oplus P_v \ar[d]\ar@{->}[d]_{\xi_1={\begin{pmatrix} id &0 \\0 & xy \end{pmatrix} }}  \ar[r]^{\begin{pmatrix} x^{\ovr{t-u}} \\ y^{\ovr{v-t}} \end{pmatrix}}   & P_t \ar[d]^x \\
P_u \oplus P_v  \ar[r]^{\begin{pmatrix} x^{\ovr{t-u}+1} \\ y^{\ovr{v-t}-1} \end{pmatrix}}   & P_{t+1} \\
}
&
\xymatrix@C=70pt@R=30pt{
P_u \oplus P_v   \ar[r]^{\begin{pmatrix} x^{\ovr{t-u}} \\ y^{\ovr{v-t}} \end{pmatrix}}   & P_t  \\
P_u \oplus P_v \ar[u]^{\xi_2={\begin{pmatrix} xy &0 \\0 & id \end{pmatrix} }} \ar[r]^{\begin{pmatrix} x^{\ovr{t-u}+1} \\ y^{\ovr{v-t}-1} \end{pmatrix}}   & P_{t+1} \ar[u]_y \\
}
\end{array}
\]
\end{defn}
\begin{itemize}
\item \textbf{A} Adjacent maps. \\

For each arc $J$, as defined at the beginning of $\S \ref{sec:catalog of maps}$, there is a sequence of maps, whose properties depend on the relative sizes of $r$ and $p$ and the location of $I_0 \subset I$. The maps we will give are all well-defined and not homotopic to zero.

Case 1: If $\mid J \mid \geq 2(p-1)$, maps as in Diagram 1.

Case 2: If $p-1 \leq  \mid J \mid < 2(p-1)$, maps as in Diagram 2.

Case 3: If $1 \leq \mid J \mid  < p-1$, maps as in Diagram 3.

Diagram 1:  $j = |J|+1$;  $j > 2(p-1)$
\[\xymatrix@C=70pt@R=20pt{
P_{v_{i-1}} \ar@{.}[d]     \\
P_{u-1}  \ar@/^/[d]^{x}	\\
P_u   \ar@/^/[u]^y  \ar@(l,l)[ddd]_{xy} 	\\
P_u \ar[u]_{\id} \ar[r]^x  \ar[d]^{\id}  &  P_{u+1} \ar[d]^x	\\
P_u \ar@/^/[u]^{xy} \ar[r]^{x^2} \ar@{.}[d]          &  P_{u+2}  \ar@/^/[u]^y \ar@{.}[d]	\\
P_u \ar[r]^{x^{p-1}}  &  P_{u+(p-1)} \ar[d]^x 	\\
 &  P_{u+p} \ar@/^/[u]^y \ar@{.}[d] 	\\
P_v \ar[r]^{y^{p-1}} \ar[d]^{xy} &  P_{v-(p-1)} \ar[d]^x \\
P_v \ar[r]^{y^{p-2}} \ar@/^/[u]^{\id} \ar@{.}[d] &  P_{v-(p-2)} \ar@/^/[u]^y \ar@{.}[d] \\
P_v \ar[r]^{y} \ar@/^/[d]^{\id} &  P_{v-1} \\
P_{v} \ar@/^/[d]^x   \ar@(l,l)[uuu]^{xy}  \\
P_{v+1} \ar@/^/[u]^y \ar@{.}[d] \\
P_{u_{i+1}}
}\]

\newpage
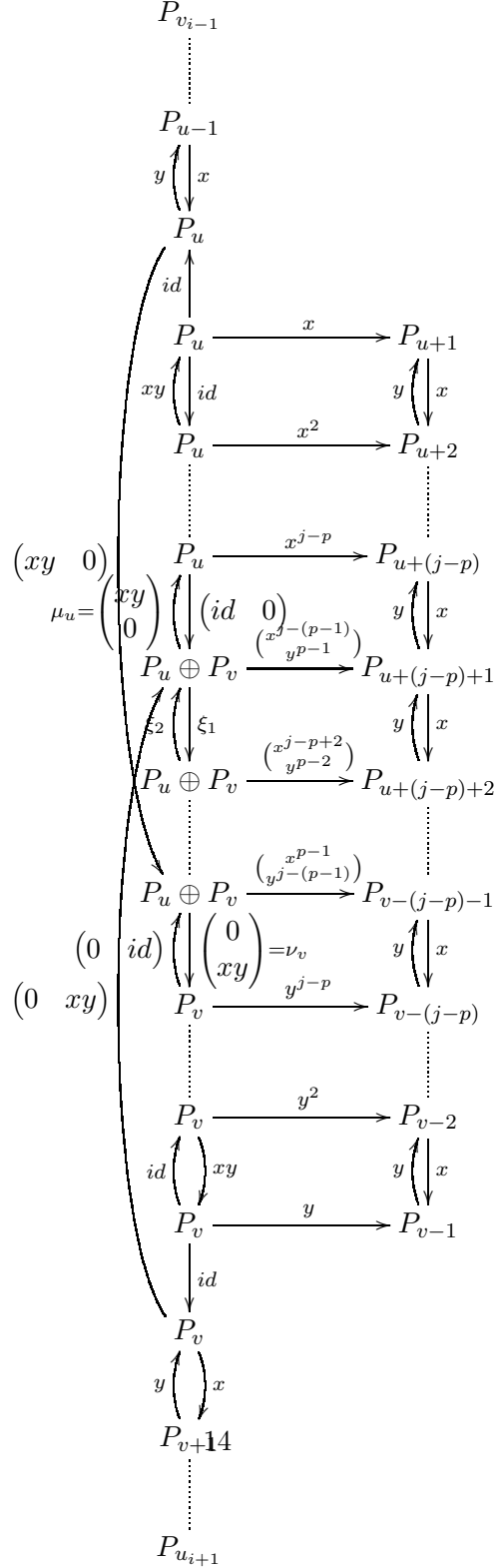
Diagram 2:  $j = |J|+1$;  $p-1 <  j  \leq 2(p-1)$
\[\xymatrix@C=40pt@R=25pt{
	P_{v_{i-1}} \ar@{.}[d]     \\
	P_{u-1}  \ar[d]^x	\\
	P_u  \ar@/^/[u]^y  \ar@(l,l)[dddddd]_{\begin{pmatrix} xy & 0 \end{pmatrix} }	\\
	P_u \ar[u]^{id} \ar[r]^x  \ar[d]^{id}  &  P_{u+1} \ar[d]^x	\\
	P_u \ar@/^/[u]^{xy} \ar[r]^{x^2} \ar@{.}[d]          &  P_{u+2}  \ar@/^/[u]^y \ar@{.}[d]	\\
	P_u \ar[r]^{x^{j-p}}  \ar[d]^{\begin{pmatrix} id& 0 \end{pmatrix}}  &  P_{u+(j-p)} \ar[d]^x 	\\
	P_u \oplus P_v \ar@/^/[u]^{\mu_u =\begin{pmatrix} xy \\ 0 \end{pmatrix}} \ar[r]^{\binom{x^{j-(p-1)}}{y^{p-1}}}  \ar[d]^{\xi_1}  &  P_{u+(j-p)+1} \ar@/^/[u]^y \ar[d]^x  \\
	P_u \oplus P_v \ar@/^/[u]^{\xi_2} \ar[r]^{\binom{x^{j-p+2}}{y^{p-2}}} \ar@{.}[d] &  P_{u+(j-p)+2} \ar@/^/[u]^y \ar@{.}[d]  \\
	P_u \oplus P_v \ar[r]^{\binom{x^{p-1}}{y^{j-(p-1)}}}  \ar[d]^{\begin{pmatrix} 0 \\ xy \end{pmatrix}=\nu_v }  &  P_{v-(j-p)-1} \ar[d]^x  \\
	P_v \ar@/^/[u]^{\begin{pmatrix} 0 & id \end{pmatrix} } \ar[r]^{y^{j-p}} \ar@{.}[d] &  P_{v-(j-p)} \ar@/^/[u]^y \ar@{.}[d] \\
	P_v \ar[r]^{y^{2}} \ar@/^/[d]^{xy} &  P_{v-2} \ar[d]^x \\
	P_v \ar@/^/[u]^{id}\ar[r]^{y} \ar[d]^{id} &  P_{v-1} \ar@/^/[u]^{y}\\
	P_{v} \ar@/^/[d]^x   \ar@(l,l)[uuuuuu]^{\begin{pmatrix} 0 & xy \end{pmatrix} }  \\
	P_{v+1} \ar@/^/[u]^y \ar@{.}[d]  \\
	P_{u_{i+1}}
}\]

If $j=p$, the horizontal maps $x, x^2 \dots x^{j-p}$ as well as $y, y^2 \dots y^{j-p}$ vanish and we get the left map of Diagram 3.

Diagram 3:  $j = |J|+1$; $2 \leq j \leq p-1$

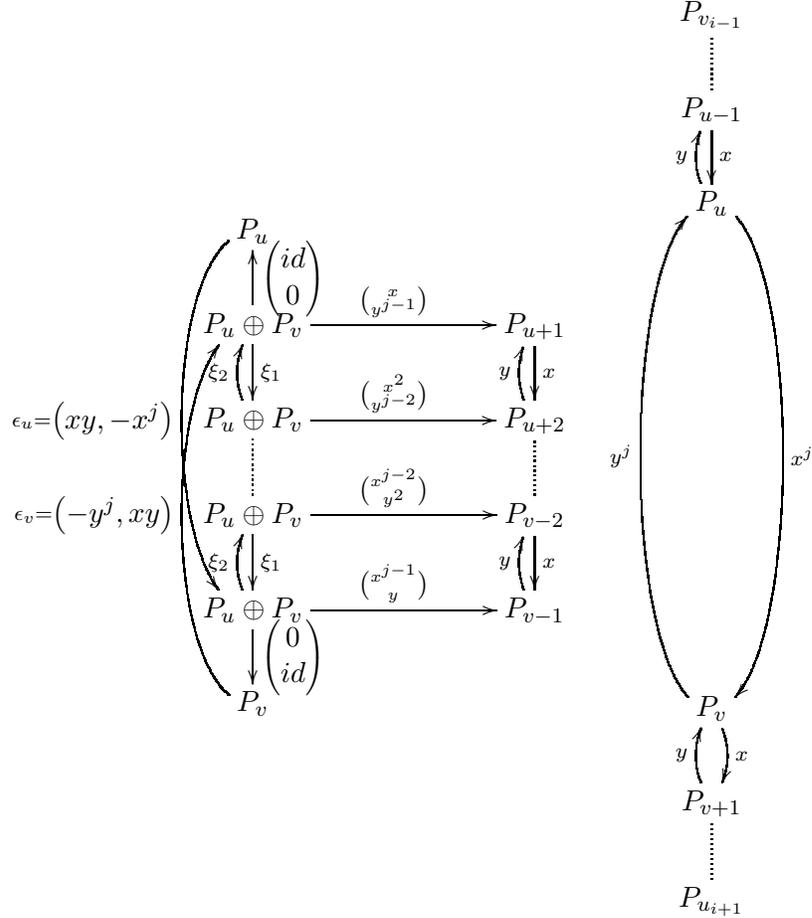
\[\xymatrix@C=70pt@R=20pt{
\\
\\
\\
P_u   \ar@(l,l)[dddd]_{\epsilon_u=\begin{pmatrix} {xy,-x^{j}}\end{pmatrix} }	\\
P_u \oplus P_v  \ar[u]_{\begin{pmatrix} id\\ 0 \end{pmatrix} } \ar[r]^{\binom{x^{}}{y^{j-1}}}   \ar[d]^{\xi_1}  &  P_{u+1} \ar[d]^x	\\
P_u \oplus P_v  \ar@/^/[u]^{\xi_2} \ar[r]^{\binom{x^{2}}{y^{j-2}}} \ar@{.}[d]   &  P_{u+2}  \ar@/^/[u]^y \ar@{.}[d]	\\
P_u \oplus P_v \ar[r]^{\binom{x^{j-2}}{y^{2}}}    \ar[d]^{\xi_1}  &  P_{v-2} \ar[d]^{x} 	\\
P_u \oplus P_v  \ar@/^/[u]^{\xi_2}  \ar[r]^{\binom{x^{j-1}}{y^{}}}    \ar[d]^{\begin{pmatrix} 0\\ id \end{pmatrix} }  &  P_{v-1}  \ar@/^/[u]^y	\\
P_v     \ar@(l,l)[uuuu]^{\epsilon_v=\begin{pmatrix} -y^{j}, xy\end{pmatrix} }
 \\
}
\quad
\xymatrix@C=70pt@R=20pt{
P_{v_{i-1}} \ar@{.}[d]     \\
P_{u-1}  \ar[d]^{x}  \ar@{.}[d]	\\
P_u  \ar@/^/[u]^y   \ar@(r,r)[ddddddd]^{x^j}	\\
\\
\\
\\
\\
\\
\\
P_{v} \ar@/^/[d]^x   \ar@(l,l)[uuuuuuu]^{y^j}    \\
P_{v+1} \ar@/^/[u]^y \ar@{.}[d] \\
P_{u_{i+1}}
}
\]
We recall that we denoted $T_i$ in Definition \ref{Elementary} by $P'_i$.
For $i, i+1 \in J$ we denote by $\alpha_{i,i+1}: P'_i \to P'_{i+1}$ the appropriate adjacent map from the diagrams, and by $\alpha_{i+1,i}: P'_{i+1} \to P'_{i}$ the corresponding adjacent map in the opposite direction.

 We denote by $\alpha_{t,\tilde{t}}$ the composition of maps  $\alpha_{\tt-1,\tt}$ $\circ \dots$ $\circ$ $\alpha_{{t}+1,{t}+2} \circ$  $\alpha_{t,t+1}$  and by $\alpha_{\tt,t}$ the composition $\alpha_{t+1,t} \circ \dots \circ \alpha_{\tilde{t}-1,\tilde{t}-2} \circ$ $\alpha_{\tilde{t}, \tilde{t}-1}$.
 	
 	 For $t, t+1 \in I_0$ we denote by $\gamma_{t,t+1}: P_t \to P_{t+1}$ and by $\gamma_{t+1,t}: P_{t+1} \to P_{t}$ multiplication by $x$ or $y$ respectively. When  $|I_0^i|=1$ or $|I_0^{i+1}|=1$, we have a map $\epsilon_t: P'_{t} \to P'_{t}$ inherited from the self-map  $xy$ from $P_t$ to itself.

\item  \textbf{B}  In addition, we have another set of self-maps which are maps from $Q_t \to P_t$ to itself which are not homotopic to zero.
For $t \in J$,
\[
\textbf{(B.1)}
\begin{array}{cc}
\xymatrix@C=50pt@R=30pt{
P_u \oplus P_v \ar[d]\ar@{->}[d]_{\begin{pmatrix} -d_0(xy)^{\ovr{v-t}} &d_0x^{\ovr{v-u}} \\ 0&0 \end{pmatrix} }  \ar[r]^{\begin{pmatrix} x^{\ovr{t-u}} \\ y^{\ovr{v-t}} \end{pmatrix}}   & P_t \ar[d]^0 \\
P_u \oplus P_v  \ar[r]^{\begin{pmatrix} x^{\ovr{t-u}} \\ y^{\ovr{v-t}} \end{pmatrix}}   & P_{t} 
}
\end{array}
\]
\[
\textbf{(B.2)}
\begin{array}{cc}
\xymatrix@C=50pt@R=30pt{
P_u \oplus P_v \ar[d]\ar@{->}[d]_{\begin{pmatrix} 0 &0 \\ c_0y^{\ovr{v-u}}&-c_0(xy)^{\ovr{t-u}} \end{pmatrix} }  \ar[r]^{\begin{pmatrix} x^{\ovr{t-u}} \\ y^{\ovr{v-t}} \end{pmatrix}}   & P_t \ar[d]^0 \\
P_u \oplus P_v  \ar[r]^{\begin{pmatrix} x^{\ovr{t-u}} \\ y^{\ovr{v-t}} \end{pmatrix}}   & P_t 
}
\end{array}
\]
We will show in the main theorem that when $d_0+c_0=0$, the sum of B.1 and B.2 is homotopic to zero. We have left the dual pair in the catalog for convenience in decomposing other maps.
\item \textbf{C}. Additional self-maps not necessarily irreducible

\textbf{(C.1)} Case of single projective in $Q_t$: By Lemma \ref{homogeneous map lem}, a complex with $Q_t$ a single projetive can exist anly if $j>p$.
 It is usually a composition of $\alpha_{t,t\pm 1}$ and $\alpha_{t\pm 1,t}$ but it is convenient to refer to, and is actually necessary to the catalog when $j=p+1$.
\[
\begin{array}{cc}
\xymatrix@C=70pt@R=30pt{
P_u\ar[d]\ar@{->}[d]_{\delta_{u,t}^q:(xy)^q} \ar[r]^{x^{\ovr{t-u}}}   & P_t \ar[d]^{(xy)^q} \\
P_u \ar[r]^{x^{\ovr{t-u}}} & P_{t} 
}
&
\xymatrix@C=70pt@R=30pt{
P_v \ar[d]\ar@{->}[d]_{\delta_{v,t}^w:(xy)^w} \ar[r]^{y^{\ovr{v-t}}}   & P_t \ar[d]^{(xy)^w} \\
P_v \ar[r]^{y^{\ovr{v-t}}} & P_{t} 
}
\end{array}
\]
The chain map $\delta_{u,t}^q$ is homotopic to zero when $q \geq \ovr{u-t}$ and 
$\delta_{v,t}^w$ is homotopic to zero when $w \geq \ovr{v-t}$.

.

\textbf{(C.2)} Case of double projective in $Q_t$: $j \leq 2(p-1)$
\[
\begin{array}{cc}
\xymatrix@C=30pt@R=30pt{
P_u \oplus P_v \ar[d]\ar@{->}[d]_{\begin{pmatrix} (xy)^q &0 \\& (xy)^q \end{pmatrix} }  \ar[r]^{\begin{pmatrix} x^{\ovr{t-u}} \\ y^{\ovr{v-t}} \end{pmatrix}}   & P_t \ar[d]^{(xy)^q} \\
P_u \oplus P_v  \ar[r]^{\begin{pmatrix} x^{\ovr{t-u}} \\ y^{\ovr{v-t}} \end{pmatrix}}   & P_{t} 
}
\end{array}
\]
As in the case of C.1, this is usually a composition of  $\alpha_{t,t\pm 1}$ and $\alpha_{t\pm 1,t}$ which we include as a convenience, but it is actually irreducible and needed in the catalog when $j=2$. In order to consider the homotopy relations, it is useful to consider the following linear combination of B.1, B.2 and C.2:
\[
\begin{array}{cc}
	\xymatrix@C=30pt@R=30pt{
		P_u \oplus P_v \ar[d]\ar@{->}[d]_{\begin{pmatrix} (h_0-d_0)(xy)^q &d_0(xy)^{q-(\ovr{v-t})}x^{\ovr{v-u}} \\c_0(xy)^{q-(\ovr{t-u})}y^{\ovr{v-u}}& (h_0-c_0)(xy)^q \end{pmatrix} }  \ar[r]^{\begin{pmatrix} x^{\ovr{t-u}} \\ y^{\ovr{v-t}} \end{pmatrix}}   & P_t \ar[d]^{h_0(xy)^q} \\
		P_u \oplus P_v  \ar[r]^{\begin{pmatrix} x^{\ovr{t-u}} \\ y^{\ovr{v-t}} \end{pmatrix}}   & P_{t} 
	}
\end{array}
\]
If $q<\ovr{v-t}$ then $d_0=0$ and if $q<\ovr{t-u}$ then $c_0=0$. The map is homotopic to zero only if $b_0+c_0=h_0$ and then the homotopy is 
\[
T=(c_0(xy)^{q-(\ovr{t-u})}y^{\ovr{t-u}}, d_0(xy)^{q-(\ovr{v-t})}x^{\ovr{v-t}} ).
\]

\textbf{(C.3)} Case of double projective in $Q_t$: $j =\ovr{v-u} \leq r$ and $\ovr{u-v}=r-j.$
\[
\begin{array}{cc}
\xymatrix@C=50pt@R=30pt{
P_u \oplus P_v \ar[d]\ar@{->}[d]_{\begin{pmatrix}0 &e_0y^{\ovr{u-v}+hr}(xy)^{q_2} \\f_0x^{\ovr{u-v}+hr}(xy)^{q_1}& 0 \end{pmatrix} }  \ar[r]^{\begin{pmatrix} x^{\ovr{t-u}} \\ y^{\ovr{v-t}} \end{pmatrix}}   & P_t \ar[d]^0 \\
P_u \oplus P_v  \ar[r]^{\begin{pmatrix} x^{\ovr{\tilde{t}-u}} \\ y^{\ovr{v-\tilde{t}}} \end{pmatrix}}   & P_{\tilde{t}} 
}
\end{array}
\]
For $e_0 \neq 0$, the map is non-zero if $ 0\leq q_1 \leq p-1-\ovr{u-v}-hr$ and  well-defined if $q_1 \geq p-\ovr{\tt-v}-hr$. For $f_0 \neq 0$, it is non-zero if $0 \leq q_2 \leq p-1-\ovr{u-v}-hr$ and well-defined if $q_2 \geq p-\ovr{u-\tt}-hr$. These occur for the maximal values of $h$

\item \textbf{D}. When $r<p$, self-maps with double projectives.
\[\textbf{(D.1)}
\begin{array}{cc}
	\xymatrix@C=50pt@R=30pt{
		P_u \oplus P_v \ar[d]\ar@{->}[d]_{\begin{pmatrix} 
			x^r&0\\
				0&x^r \end{pmatrix} }  \ar[r]^{\begin{pmatrix} x^{\ovr{t-u}}\\ y^{\ovr{v-t}} \end{pmatrix}}   & P_t \ar[d]^{x^r} \\
		P_u \oplus P_v  \ar[r]^{\begin{pmatrix} x^{\ovr{t-u}} \\ y^{\ovr{v-t}} \end{pmatrix}}   & P_{{t}} 
	}
\end{array}
\]

\[\textbf{(D.2)}
\begin{array}{cc}
	\xymatrix@C=50pt@R=30pt{
		P_u \oplus P_v \ar[d]\ar@{->}[d]_{\begin{pmatrix} 
				y^r&0\\
				0&y^r \end{pmatrix} }  \ar[r]^{\begin{pmatrix} x^{\ovr{t-u}}\\ y^{\ovr{v-t}} \end{pmatrix}}   & P_t \ar[d]^{y^r} \\
		P_u \oplus P_v  \ar[r]^{\begin{pmatrix} x^{\ovr{t-u}} \\ y^{\ovr{v-t}} \end{pmatrix}}   & P_{{t}} 
	}
\end{array}
\]

\textbf{(D.3)} This is a special case, occuring only when $r=2j <p$ and $h$ is maximal with $hr <p$. We combine two linearly independent maps. The map is well-defined for $e_0 \neq 0$ if $hr+\ovr{v-t} \geq p$ and well-defined for $f_0 \neq 0$ if $hr+\ovr{t-u} \geq p$.
\[ 
\begin{array}{cc}\\
	\xymatrix@C=50pt@R=30pt{
		P_u \oplus P_v \ar[d]\ar@{->}[d]_{\begin{pmatrix} 
				0&e_0 y^{hr-j}(xy)^{\ovr{t-u}}\\
				f_0 x^{hr-j}(xy)^{\ovr{v-t}}&0 \end{pmatrix} }  \ar[r]^{\begin{pmatrix} x^{\ovr{t-u}}\\ y^{\ovr{v-t}} \end{pmatrix}}   & P_t \ar[d]^{e_0y^{hr}+f_0x^{hr}}
			 \\
		P_u \oplus P_v  \ar[r]^{\begin{pmatrix} x^{\ovr{t-u}} \\ y^{\ovr{v-t}} \end{pmatrix}}   & P_{{t}} 
	}
\end{array}
\]

\item $E$. Chain maps with one vertical term zero.

\textbf{(E.1)} Chain maps with degree $0$ term zero. Case of single projective in $Q_t$: which can exist just if $j>p$.

\[
\begin{array}{cc}
	\xymatrix@C=70pt@R=30pt{
		P_u\ar[d]\ar@{->}[d]_{0} \ar[r]^{x^{\ovr{t-u}}}   & P_t \ar[d]^{x^{\tilde{t}-t}(xy)^{p-\overline{t-u`}}} \\
		Q_{\tilde{t}} \ar[r] & P_{\tilde{t}} 
	}
	&
	\xymatrix@C=70pt@R=30pt{
		P_v \ar[d]\ar@{->}[d]_{0} \ar[r]^{y^{\ovr{v-\tilde{t}}}}   & P_{\tilde{t}} \ar[d]^{y^{\tilde{t}-t}(xy)^{p-\ovr{v-t}}} \\
		Q_t \ar[r] & P_{t} 
	}
\end{array}
\]

\textbf{(E.2)}
Chain maps with degree $1$ term zero, single projective. We assume $j >p$. For every $w \in I_0^i$,$t \in J$, such that $\ovr{t-u} \leq j-p$, $\ovr{t-w}\leq p$, we have the following, which we denote by $\beta^u_w: P'_w \to P'_{t}$.
\[
\begin{array}{cc}
	\xymatrix@C=70pt@R=30pt{
		P_w\ar[d]\ar@{->}[d]_{x^{\ovr{u-w}}(xy)^{p-(\ovr{t-w})}}  \\
		P_u \ar[r]^{x^{\ovr{t-u}}} & P_{t} 
	}
\end{array}
\]
For every $w \in I_0^{i+1},t \in J$, such that $\ovr{v-t} \leq j-p$, $\ovr{w-t}\leq p$, 
we have the following map, which  we denote by $\beta^v_w: P'_w \to P'_{t}$.
\[
\begin{array}{cc}
	\xymatrix@C=70pt@R=30pt{
		P_w\ar[d]\ar@{->}[d]_{y^{\ovr{w-v}}(xy)^{p-(\ovr{w-t})}}  \\
		P_v \ar[r]^{y^{\ovr{v-t}}} & 																																				P_{t} 
	}
\end{array}
\]

\textbf{(E.3)}
Chain maps with degree  $1$ term zero, double projective. For $hr < p$ and every $t \in J$, such that $hr +\ovr{t-u} \geq p$, or in the dual, $hr +\ovr{v-t} \geq p$.

\[
\begin{array}{cc}
	\xymatrix@C=50pt@R=30pt{
		P_u \oplus P_v \ar[d]\ar@{->}[d]_{\begin{pmatrix} 
			x^{hr}&0\\
			0&0\end{pmatrix} }  \ar[r]^{\begin{pmatrix} x^{\ovr{t-u}}\\ y^{\ovr{v-t}} \end{pmatrix}}   & P_t \ar[d]^{0} \\
		P_u \oplus P_v  \ar[r]^{\begin{pmatrix} x^{\ovr{\tilde{t}-u}} \\ y^{\ovr{v-\tilde{t}}} \end{pmatrix}}   & P_{\tilde{t}} 
	}
\end{array}
\]
\[
\begin{array}{cc}
	\xymatrix@C=50pt@R=30pt{
		P_u \oplus P_v \ar[d]\ar@{->}[d]_{\begin{pmatrix} 
				0&0\\
				0&y^{hr}\end{pmatrix} }  \ar[r]^{\begin{pmatrix} x^{\ovr{t-u}}\\ y^{\ovr{v-t}} \end{pmatrix}}   & P_t \ar[d]^{0} \\
		P_u \oplus P_v  \ar[r]^{\begin{pmatrix} x^{\ovr{\tilde{t}-u}} \\ y^{\ovr{v-\tilde{t}}} \end{pmatrix}}   & P_{\tilde{t}} 
	}
\end{array}
\]

\textbf{(E.4)} Chain maps with degree  $1$ term zero, stalk complex to double  projective.
\[
\begin{array}{cc}
	\xymatrix@C=50pt@R=30pt{
		P_u \ar[d]\ar@{->}[d]_{\begin{pmatrix} 
			x^{hr}&0
			\end{pmatrix} }   \\
		P_u \oplus P_v  \ar[r]^{\begin{pmatrix} x^{\ovr{{t}-u}} \\ y^{\ovr{v-{t}}} \end{pmatrix}}   & P_{{t}} 
	}
\end{array}
\]

Well-defined if $hr+\ovr{t-u}\geq p$.
\[
\begin{array}{cc}
	\xymatrix@C=50pt@R=30pt{
		P_v \ar[d]\ar@{->}[d]_{\begin{pmatrix} 
				0&y^{hr}\end{pmatrix} }  \\
		P_u \oplus P_v  \ar[r]^{\begin{pmatrix} x^{\ovr{t-u}} \\ y^{\ovr{v-t}} \end{pmatrix}}   & P_{t} 
	}
\end{array}
\]

Well-defined if $hr+\ovr{v-t}\geq p$. 

\textbf{(E.5)} Maps of stalk complexes. Suppose the $I-I_0$ contains more than a single connected component, that $r<p$ and  that $s=\ovr{v_{i}-u_{i+1}}<p$. Maps from $P_{u_{i+1}}$  $P_{v_{i-1}}$ given by $x^s$ or $y^s$  must be included in the catalog in order that we can obtain self maps of $P_w$ given by $x^r$ and $y^r$. These are similar to the maps $x^j$ and $y^j$ appearing in Diagram 3.  If the interval in $I_0$ containing $w$ consists of a single index, then the self-map of $P_w$ given by $xy$ is irreducible.

\end{itemize}

\section{Maps between components of an elementary tilting complex} 
In this section, we will prove the proposition justifying the listed maps given in the catalog. In $\S \ref{catalog of intra-arc maps}$ the adjacent maps in $A$ are irreducible, as are the self-maps in $B$. Of the maps in $C$, the maps with the smallest powers of $xy$ are sometimes irreducible. When $u+1=v-1=t$, the map $C.1$ with $q=1$ is irreducible.
 
\begin{prop}\label{internal} We are given a group algebra $F[C_p \times C_p)) \rtimes C_r]$ and  an elementary equivalence determined by a subset $I_0$ of the set of residues $I$.
   For each interval $J_i$ in $I-I_0$, the set of maps in the catalog, whose nature  depends on the relative sizes of $r$ and $p$ and the location of $I_0$, are all well-defined and not  homotopic to zero. All the maps internal to the set $I_0^i \cup J_i \cup I_0^{i+1}$ can be written a linear combinations of compositions of maps in the catalog.
\end{prop}

\begin{proof}
\begin{enumerate}
\item  \textbf{The maps in the catalog are well-defined and not homotopic to zero}

We frequently use $x^p=y^p=0$ as we mentioned before.
All the maps between two stalk complexes are well-defined by the definition of the Brauer correspondent. Also, all the maps with zero at one end of the main diagonal are well-defined.

\textbf{A. Diagram 1}
 We want to prove that all maps in Diagram 1 are well-defined and to check homotopy to zero. Maps between stalk complexes are already well-defined. Now we will check all maps between stalk complexes and $2$-complexes as follows:

 The long arrow.
\[
\begin{array}{cc}
\xymatrix@C=60pt@R=20pt{
P_u \ar[d]_{xy}   \    \\
P_u  \ar[r]^{x^{p-1}}   &  P_{u+(p-1)}\\
}
&
\xymatrix@C=60pt@R=20pt{
P_v \ar[d]_{xy}      \\
P_v \ar[r]^{y^{p-1}}   &  P_{v-(p-1)}\\
}
\end{array}
\]

both maps are well-defined because \begin{center}  $(xy)\cdot x^{p-1}=y\cdot x^{p} = 0$ and
 $(xy)\cdot y^{p-1}=x\cdot y^{p} = 0$. \end{center} These maps are not homotopic to zero because they end with zero at the upper end of the off-diagonal.

The following maps occur only in Diagram 1:
\[
\begin{array}{cc}
\xymatrix@C=60pt@R=20pt{
        &  P_{u+p}  \ar[d]^y\\
P_u  \ar[r]^{x^{p-1}}   &  P_{u+{(p-1)}}	\\
}
&
\xymatrix@C=60pt@R=20pt{
    &  P_{v-p}  \ar[d]^x  \\
P_v  \ar[r]^{y^{p-1}}   &  P_{v-{(p-1)}}	\\
}
\end{array}
\]

both maps are well-defined because one of the modules on the main diagonal is zero.
These maps are not homotopic to zero because the diagonal maps are zero.

The following maps occur only in Diagram 1:
\[
\begin{array}{cc}
\xymatrix@C=60pt@R=20pt{
P_u  \ar[r]^{x^{p-1}}         &  P_{u+(p-1)}  \ar[d]^x\\
   &  P_{u+p}	\\
}
&
\xymatrix@C=60pt@R=20pt{
P_v \ar[r]^{y^{p-1}} &  P_{v-(p-1)} \ar[d]^y\\
&  P_{v-p}  \\
}
\end{array}
\]

These maps are well-defined because $x^p=y^p=0$ and they are not homotopic to zero because one of the off-diagonal modules is the zero module.
\\

This pair of dual maps occurs also in Diagram 2.
\[
\begin{array}{cc}
\xymatrix@C=60pt@R=20pt{
P_u \ar[d]_{id} \ar[r]^{x}         &  P_{u+1}  \\
P_u 	\\
}
&
\xymatrix@C=60pt@R=20pt{
P_v \ar[r]^{y} \ar[d]^{id} &  P_{v-1} \\
P_v \\
}
\end{array}
\]

Both maps are automatically well-defined because one main diagonal module is zero and they are not homotopic to zero because of the identity map.

Now  we will consider all maps between two $2$-complexes in Diagram 1. These occur also in Diagram 2 when $j>p+1$. We start with $\alpha_{t, \ovr{t+1}}$ and $\alpha_{ \ovr{t+1},t}$ for $Q_t =P_u$ and dually, for $Q_t=P_v$, where $t, \ovr{t+1} \in J$

\[
\begin{array}{cc}
\xymatrix@C=60pt@R=20pt{
P_u \ar[d]_{id} \ar[r]^{x^{\ovr{t-u}}}         &  P_{t}  \ar[d]^x\\
P_u  \ar[r]^{x^{\ovr{t-u+1}}}   &  P_{t+1}	\\
}
&
\xymatrix@C=60pt@R=20pt{
P_v \ar[r]^{y^{\ovr{v-t-1}}} \ar[d]^{id} &  P_{\ovr{t+1}} \ar[d]^y\\
P_v \ar[r]^{y^{\ovr{v-t}}}&  P_{t}  \\
}
\end{array}
\]
\noindent They are well-defined because $ x^{\ovr{t-u+1}} \circ id=x \circ x^{\ovr{t-u}}$ and dually $ y^{\ovr{v-t}} \circ id= y \circ y^{\ovr{v-t-1} } $. Maps that have a square containing a vertical identity can never be homotopic to zero.
In the second case:

\[
\begin{array}{cc}
\xymatrix@C=60pt@R=20pt{
P_u \ar[d]^{xy} \ar[r]^{x^{\ovr{t-u+1}}}         &  P_{t+1}  \ar[d]^y 	\\
P_u  \ar[r]^{x^{\ovr{t-u}}}   &  P_{t}	\\
}
&
\xymatrix@C=60pt@R=20pt{
P_v \ar[r]^{y^{\ovr{v-t}}} \ar[d]^{xy} &  P_t \ar[d]^x \\
P_v \ar[r]^{y^{\ovr{v-t-1}}}&  P_{t+1}  \\
}
\end{array}
\]
the chain maps are well-defined because $ x^{\ovr{t-u}} \circ xy= y \circ x^{\ovr{t-u+1}}$ and dually
 $(xy) \cdot y^{\ovr{v-t-1}}=y^{\ovr{v-t}}\cdot x$.
We want to prove that these maps are not homotopic to zero.  If we can find $T:P_{\ovr{t+1}} \rightarrow P_{u}$, then since the horizontal map $x^{\ovr{t-u}}$ has a non-zero power of $x$, the composition with $T$ does not give $y$, and the case on the right is dual.

\textbf{A. Diagram 2}. Now we will look at Diagram $2$ and consider maps that do not occur in Diagram $1$. We start with the long arrows in Diagram 2.
\[
\begin{array}{cc}
\xymatrix@C=70pt@R=30pt{
P_u\ar[d]\ar@{->}[d]_{\begin{pmatrix} xy & 0 \end{pmatrix} }   \\
P_u \oplus P_v \ar[r]^{\begin{pmatrix} x^{p-1} \\ y^{\ovr{j-p+1}}  \end{pmatrix}} & P_{\ovr{v-j+p-1}}
}
&
\xymatrix@C=70pt@R=30pt{
P_v \ar[d]\ar@{->}[d]_{\begin{pmatrix} 0  & xy \end{pmatrix} }   \\
P_u \oplus P_v \ar[r]^{\begin{pmatrix} x^{\ovr{j-p+1}} \\ y^{p-1} \end{pmatrix}} & P_{\ovr{u+j-p+1}}
}
\end{array}
\]
 
They are well-defined because $ x^{p-1} \circ xy=0$ and dually $ y^{p-1} \circ xy=0$. Maps that have a zero module at one end of the off-diagonal can never be homotopic to zero.

\[
\begin{array}{cc}
\xymatrix@C=60pt@R=20pt{
P_u \oplus P_v \ar[r]^{\binom{x^{p-1}}{y^{\ovr{j-p+1}}}}  \ar[d]_{\begin{pmatrix} 0 \\ xy \end{pmatrix} }  &  P_{\ovr{v-j+p-1}} \ar[d]^x \ar@{->}[ld]_-{T_1} \\
P_v  \ar[r]^{y^{j-p}}  &  P_{\ovr{v-j+p}}  \\
}
&
\xymatrix@C=60pt@R=20pt{
P_u \oplus P_v \ar[d]_{\begin{pmatrix} xy \\ 0 \end{pmatrix} } \ar[r]^{\binom{x^{\ovr{j-p+1}}}{y^{p-1}}}  &  P_{\ovr{u+j-p+1}} \ar[d]^y  \ar@{->}[ld]_-{T_2} \\
P_u \ar[r]^{x^{j-p}}  &  P_{\ovr{u+j-p}} 	\\
}
\end{array}
\]
In this case  $j>p$. The chain maps are well-defined because $(xy) \cdot y^{j-p}=xy^{j-p+1}$ and dually $(xy) \cdot x^{j-p}=yx^{j-p+1}$.  If $j>p$ the maps are not homotopic to zero because if $T_1=x^{j-p+1}$ then $x^{j-p+1}y^{j-p}\not =x$ and dually, if $T_2=y^{j-p+1}$, then $y^{j-p+1}x^{j-p} \not =y$.

\[
\begin{array}{cc}
\xymatrix@C=60pt@R=20pt{
P_u \ar[d]_{\begin{pmatrix} id  & 0 \end{pmatrix} }  \ar[r]^{x^{j-p}}         &  P_{\ovr{u+j-p}}  \ar[d]^x\\
P_u\oplus P_v  \ar[r]^{\binom{x^{\ovr{j-p+1}}}{y^{p-1}}}   &  P_{\ovr{u+j-p+1}}	\\
}
&
\xymatrix@C=60pt@R=20pt{
P_v \ar[r]^{y^{j-p}} \ar[d]_{\begin{pmatrix} 0  & id \end{pmatrix} }  &  P_{\ovr{v-j+p}} \ar[d]^y\\
P_u\oplus P_v \ar[r]^{\binom{x^{p-1}}{y^{\ovr{j-p+1}}}} &  P_{\ovr{v-j+p-1}}  \\
}
\end{array}
\]
Again  $j>p$. The chain maps are well-defined because $x^{\ovr{j-p}}\cdot x=x^{j-p+1}$ and dually $y^{j-p}\cdot y=y^{j-p+1}$. Maps that have  a square containing a vertical identity map can never be homotopic to zero.

Let $j=\ovr{v-u}$ and $t \in J$.

\[
\begin{array}{cc}
\xymatrix@C=70pt@R=30pt{
P_u \oplus P_v \ar[d]\ar@{->}[d]_{\xi_1={\begin{pmatrix} id &0 \\0 & xy \end{pmatrix} }}  \ar[r]^{\begin{pmatrix} x^{\ovr{t-u}} \\ y^{\ovr{v-t}} \end{pmatrix}}   & P_t \ar[d]^x \ar@{->}[ld]_-{} \\
P_u \oplus P_v  \ar[r]^{\begin{pmatrix} x^{\ovr{t-u+1}} \\ y^{\ovr{v-t-1}} \end{pmatrix}}   & P_{t+1} \\
}
&
\xymatrix@C=70pt@R=30pt{
P_u \oplus P_v   \ar[r]^{\begin{pmatrix} x^{\ovr{t-u}} \\ y^{\ovr{v-t}} \end{pmatrix}}   & P_t  \\
P_u \oplus P_v  \ar@{->}[ru]_-{} \ar[u]^{\xi_2={\begin{pmatrix} xy &0 \\0 & id \end{pmatrix} }} \ar[r]^{\begin{pmatrix} x^{\ovr{t-u+1}} \\ y^{\ovr{v-t-1}} \end{pmatrix}}   & P_{t+1} \ar[u]_y \\
}
\end{array}
\]
They are well-defined because $\begin{pmatrix} id &0 \\0 & xy \end{pmatrix}$ $\cdot$ $\begin{pmatrix} x^{\ovr{t-u+1}} \\ y^{\ovr{v-t-1}} \end{pmatrix}$ $=$ $\begin{pmatrix} x^{\ovr{t-u}} \\ y^{\ovr{v-t}} \end{pmatrix}$ $\cdot$ $x$ as well as the dual case. They are not homotopic to zero because if in the map on the left, the first component of the homotopy is nonzero, the composition with the horizontal map contains $x^2$ and if the second component is non-zero, the composition has a power of $y$. The other map is dual.

\textbf{A. Diagram 3}. These maps occur in Diagram $3$ when $j \leq p-1$ or in Diagram $2$ when $j=p$.

\[
\begin{array}{cc}
\xymatrix@C=70pt@R=30pt{
P_u\ar[d]\ar@{->}[d]_{\begin{pmatrix} xy  & -x^j \end{pmatrix} }   \\
P_u \oplus P_v \ar[r]^{\begin{pmatrix} x^{j-1}\\ y \end{pmatrix}} & P_{v-1}
}
&
\xymatrix@C=70pt@R=30pt{
P_v \ar[d]\ar@{->}[d]_{\begin{pmatrix} -y^j  & xy \end{pmatrix} }   \\
P_u \oplus P_v \ar[r]^{\begin{pmatrix} x \\ y^{j-1} \end{pmatrix}} & P_{u+1}
}
\end{array}
\]
 
The chain maps are well-defined because $(xy) \cdot x^{j-1}-x^j \cdot y=0$ and dually $-y^j \cdot x+xy \cdot y^{j-1}=0$. Maps which have a zero module at one end of the off-diagonal can never be homotopic to zero.

\textbf{Self-maps B}.
We now check well-defined for the maps in  $B$.

For $t \in J$
\[
\begin{array}{cc}
\xymatrix@C=50pt@R=30pt{
P_u \oplus P_v \ar[d]\ar@{->}[d]_{\begin{pmatrix} -b_0(xy)^{\ovr{v-t}} &b_0x^{\ovr{v-u}} \\ 0&0 \end{pmatrix} }  \ar[r]^{\begin{pmatrix} x^{\ovr{t-u}} \\ y^{v-t} \end{pmatrix}}   & P_t \ar[d]^0 \\
P_u \oplus P_v  \ar[r]^{\begin{pmatrix} x^{\ovr{t-u}} \\ y^{\ovr{v-t}} \end{pmatrix}}   & P_{t} 
}
\end{array}
\]
\[
\begin{array}{cc}
\xymatrix@C=50pt@R=30pt{
P_u \oplus P_v \ar[d]\ar@{->}[d]_{\begin{pmatrix} 0 &0 \\ c_0y^{\ovr{v-u}}&-c_0(xy)^{\ovr{t-u}} \end{pmatrix} }  \ar[r]^{\begin{pmatrix} x^{\ovr{t-u}} \\ y^{\ovr{v-t}} \end{pmatrix}}   & P_t \ar[d]^0 \\
P_u \oplus P_v  \ar[r]^{\begin{pmatrix} x^{\ovr{t-u}} \\ y^{\ovr{v-t}} \end{pmatrix}}   & P_t 
}
\end{array}
\]
 For the first map, let $a \in P_u$ and $b \in P_v$. We check the commutativity of the square by diagram chasing.

\begin{align*}
(a,b) \cdot {\begin{pmatrix} -b_0(xy)^{\ovr{v-t}} &b_0x^{\ovr{v-u}} \\ 0&0 \end{pmatrix} }  \cdot {\begin{pmatrix} x^{\ovr{t-u}} \\ y^{\ovr{v-t}} \end{pmatrix}}=-ab_0x^{\ovr{t-u}}(xy)^{\ovr{v-t}}+ab_0x^{\ovr{v-u}}y^{\ovr{v-t}}=0
\end{align*}
For the second map
\begin{align*}
(a,b) \cdot {\begin{pmatrix} 0 &0 \\ c_0y^{\ovr{v-u}}&-c_0(xy)^{\ovr{t-u}} \end{pmatrix} } \cdot {\begin{pmatrix} x^{\ovr{t-u}} \\ y^{\ovr{v-t}} \end{pmatrix}}=bc_0x^{\ovr{t-u}}y^{\ovr{v-u}}-bc_0(xy)^{\ovr{t-u}}y^{\ovr{v-t}}=0
\end{align*}

Now we will prove that B.1 is not homotopic to zero, and we will prove while dealing with the maps of C that B.2 is homotopic to B.1.

For $t \in J$,
\[
\begin{array}{cc}
\xymatrix@C=50pt@R=30pt{
P_u \oplus P_v \ar[d]\ar@{->}[d]_{\begin{pmatrix} -d_0(xy)^{\ovr{v-t}} &d_0x^{\ovr{v-u}} \\ 0&0 \end{pmatrix} }  \ar[r]^{\begin{pmatrix} x^{\ovr{t-u}} \\ y^{\ovr{v-t}} \end{pmatrix}}   & P_t \ar[d]^0 \ar@{->}[ld]_-{T}\\
P_u \oplus P_v  \ar[r]_{\begin{pmatrix} x^{\ovr{t-u}} \\ y^{\ovr{v-t}} \end{pmatrix}}   & P_{t} 
}
\end{array}
\]
We need to find $T:P_{t} \rightarrow P_u \oplus P_v$ and to check if
\begin{align*}
 {\begin{pmatrix} x^{\ovr{t-u}} \\ y^{\ovr{v-t}} \end{pmatrix}} \circ T &=0\\
 T \circ{\begin{pmatrix} x^{\ovr{t-u}} \\ y^{\ovr{v-t}} \end{pmatrix}}  &={\begin{pmatrix} -d_0(xy)^{\ovr{v-t}} &d_0x^{\ovr{v-u}} \\ 0&0 \end{pmatrix}, } \\
\end{align*}

Let us take as our candidate for homotopy the map $T:P_{t}\rightarrow P_u \oplus P_v$  given by
${\begin{pmatrix} e_0y^{\ovr{t-u}}(xy)^{q_1} & f_0x^{\ovr{v-t}}(xy)^{q_2} \end{pmatrix} }$. By  Lemma \ref{homogeneous map lem}, if $r>p$, this is the only possibilty. From our calculations, it will be clear that even if $r<p$, multiplying the components of $T$ by various power of $x^r$ or $y^r$ could not give the desired composition.

So, if we look at the bottom triangle $ {\begin{pmatrix} x^{\ovr{t-u}} \\ y^{\ovr{v-t}} \end{pmatrix}}  \circ T \not =0$ unless $e_0+f_0=0$ and $q_1+\ovr{t-u}=q_2+\ovr{v-t}$. If $\ovr{v-t} \geq \ovr{t-u}$, we set $q_2=q$, $q_1=q+\ovr{v+u-2t}$. If $\ovr{t-u} \geq \ovr{v-t}$, we set $q_1=q$, $q_2=q+\ovr{2t-u-v}$. Now we check the second condition for the case $\ovr{t-u} \geq \ovr{v-t}$.
\[
-d_0(xy)^{\ovr{v-t}}=x^{\ovr{t-u}}e_0y^{\ovr{v-t}}(xy)^q
\]
\noindent which can only hold if $e_0=-d_0, \ovr{t-u}=\ovr{v-t}$ and $q=0$.  However, if all those conditions hold, then we must have $0=y^{\ovr{v-t}}e_0y^{\ovr{t-u}}(xy)^0$,  which cannot be true because $\ovr{t-u}+\ovr{v-t}=j <p$.  The case $\ovr{v-t} \geq \ovr{t-u}$ is similar, using the second column of the matrix.

\textbf{Self-maps C}.

Most of the additional self-maps $C.1$ are compositions of adjacent maps and therefore well-defined. If $j=2$, the maps in $C.1$ are irreducible but can be checked by inspection. In the case of a single projective, we have compositions of adjacent maps up and down or down and up, in the case of double projectives we have compositions of the maps in Definition \ref{Def. 42}. 

C.1: Case of single projective: $j>p$.
For $q<\ovr{t-u}$, $w<\ovr{v-t}, t \in J$.
\[
\begin{array}{cc}
	\xymatrix@C=70pt@R=30pt{
		P_u\ar[d]\ar@{->}[d]_{(xy)^q} \ar[r]^{x^{\ovr{t-u}}}   & P_t \ar@{->}[ld]_-{T} \ar[d]^{(xy)^q} \\
		P_u \ar[r]^{x^{\ovr{t-u}}} & P_{t} 
	}
	&
	\xymatrix@C=70pt@R=30pt{
		P_v \ar[d]\ar@{->}[d]_{(xy)^w} \ar[r]^{y^{\ovr{v-t}}}   & P_t  \ar@{->}[ld]_-{T}\ar[d]^{(xy)^w} \\
		P_v \ar[r]^{y^{\ovr{v-t}}} & P_{t} 
	}
\end{array}
\]
When $q \geq \ovr{t-u}$, we get a homotopy to the map on the left given by 
$T=(xy)^{q-\ovr{t-u}}y^{\ovr{t-u}}$.  When $w \geq \ovr{v-t}$, we get a homotopy for the second map given by $T=(xy)^{w-\ovr{v-t}}y^{\ovr{v-t}}$

C.2: Case of double projective: $j \leq 2(p-1)$
We only have to check our claim that the map below is homotopic to zero if $h_0=d_0+c_0$ and not all the parameters are zero. We assume
$c_0 \neq 0$, in which case $q \geq \ovr{t-u}$  or else $d_0 \neq 0$, in which case $q \geq \ovr{v-t}$. 
\[
\begin{array}{cc}
\xymatrix@C=30pt@R=30pt{
P_u \oplus P_v \ar[d]\ar@{->}[d]_{\begin{pmatrix} (h_0-d_0)(xy)^q &d_0(xy)^{q-(\ovr{v-t})}x^{\ovr{v-u}} \\c_0(xy)^{q-(\ovr{t-u})}y^{\ovr{v-u}}& (h_0-c_0)(xy)^q \end{pmatrix} }  \ar[r]^{\begin{pmatrix} x^{\ovr{t-u}} \\ y^{\ovr{v-t}} \end{pmatrix}}   & P_t \ar[d]^{h_0(xy)^q} \\
P_u \oplus P_v  \ar[r]^{\begin{pmatrix} x^{\ovr{t-u}} \\ y^{\ovr{v-t}} \end{pmatrix}}   & P_{t} 
}
\end{array}
\]

We define the homotopy map $T:P_t \to P_u \otimes P_v$ by
\[
(c_0y^{\ovr(t-u)}(xy)^{q-\ovr{t-u}}, d_0 x^{\ovr{v-t}}(xy)^{q-\ovr{v-t}})
\]
Then the bottom triangle yields $(c_0+d_0)(xy)^q$ and the upper triangle yields the desired matrix, with $c_0$ written for $h_0-d_0$ and $d_0$ written for $h_0-c_0$.

Self-maps D: When $r<p$, maps with double projective.
\[\textbf{(D.1)}
\begin{array}{cc}
	\xymatrix@C=50pt@R=30pt{
		P_u \oplus P_v \ar[d]\ar@{->}[d]_{\begin{pmatrix} 
			x^r&0\\
				0&x^r \end{pmatrix} }  \ar[r]^{\begin{pmatrix} x^{\ovr{t-u}}\\ y^{\ovr{v-t}} \end{pmatrix}}   & P_t \ar[d]^{x^r} \\
		P_u \oplus P_v  \ar[r]^{\begin{pmatrix} x^{\ovr{t-u}} \\ y^{\ovr{v-t}} \end{pmatrix}}   & P_{{t}} 
	}
\end{array}
\]
 well-defined because 
\[
(a,b) \cdot  \begin{pmatrix} x^{\ovr{t-u}} \\ y^{\ovr{v-t}} \end{pmatrix}\cdot x^r 
=ax^{\ovr{t-u}} \cdot x^r +by^{\ovr{v-t}} \cdot x^r\\
\]
On the other hand
\begin{align*}
(a,b) \cdot \begin{pmatrix} 
			x^r&0\\
				0&x^r \end{pmatrix} \cdot \begin{pmatrix} x^{\ovr{t-u}} \\ y^{\ovr{v-t}} \end{pmatrix}=ax^r \cdot x^{\ovr{t-u}}+bx^r \cdot y^{\ovr{v-t}}.
\end{align*}
The case of \textbf{(D.2)} is dual.

\textbf{(D.3)} This is a special case, occuring only when $r=2j <p$ and $h$ is maximal with $hr <p$.  The map is well-defined for $f_0 \neq 0$ if $hr+\ovr{t-u} \geq p$ and is well-defined for $e_0 \neq 0$ if $hr+\ovr{v-t} \geq p$.
\[ 
\begin{array}{cc}\\
	\xymatrix@C=50pt@R=30pt{
		P_u \oplus P_v \ar[d]\ar@{->}[d]_{\begin{pmatrix} 
				0&e_0 y^{hr-j}(xy)^{\ovr{t-u}}\\
				f_0 x^{hr-j}(xy)^{\ovr{v-t}}&0 \end{pmatrix} }  \ar[r]^{\begin{pmatrix} x^{\ovr{t-u}}\\ y^{\ovr{v-t}} \end{pmatrix}}   & P_t \ar[d]^{e_0y^{hr}+f_0x^{hr}}
			 \\
		P_u \oplus P_v  \ar[r]^{\begin{pmatrix} x^{\ovr{t-u}} \\ y^{\ovr{v-t}} \end{pmatrix}}   & P_{{t}} 
	}
\end{array}
\]
\noindent A homotopy for the bottom triangle would be $(f_0 x^{hr-\ovr{t-u}},e_0 y^{hr-\ovr{v-t}})$, but this would not give the proper matrix when composing the upper triangle. 

\textbf{E}. Chain maps with vertical term zero.

\textbf{(E.1)} Case of single projective in $Q_t$: By Lemma \ref{homogeneous map lem}, it can exist just if $j>p$.

\[
\begin{array}{cc}
	\xymatrix@C=70pt@R=30pt{
		P_u\ar[d]\ar@{->}[d]_{0} \ar[r]^{x^{\ovr{t-u}}}   & P_t \ar[d]^{x^{\ovr{\tilde{t}-t}}(xy)^{p-(\overline{\tt-u})}} \\
		Q_{\tilde{t}} \ar[r] & P_{\tilde{t}} 
	}
	&
	\xymatrix@C=70pt@R=30pt{
		P_v \ar[d]\ar@{->}[d]_{0} \ar[r]^{y^{\ovr{v-\tilde{t}}}}   & P_{\tilde{t}} \ar[d]^{y^{\ovr{\tilde{t}-t}}(xy)^{p-(\ovr{v-t})}} \\
		Q_t \ar[r] & P_{t} 
	}
\end{array}
\]
The maps are well-defined because all the powers of $x$, and, in the dual case, $y$, vanish and we get zero as desired.

\textbf{(E.2)}
Chain maps with degree $1$ term zero, stalk comples to single projective. 
\[
\begin{array}{cc}
	\xymatrix@C=70pt@R=30pt{
		P_w\ar[d]\ar@{->}[d]_{x^{\ovr{u-w}}(xy)^{p-(\ovr{t-w})}}  \\
		P_u \ar[r]^{x^{\ovr{t-u}}} & P_{t} 
	}
\end{array}
\]
\[
\begin{array}{cc}
	\xymatrix@C=70pt@R=30pt{
		P_w\ar[d]\ar@{->}[d]_{y^{\ovr{w-v}}(xy)^{p-(\ovr{w-t})}}  \\
		P_v \ar[r]^{y^{\ovr{v-t}}} & 																																				P_{t} 
	}
\end{array}
\]
In the first of these, the horizontal map exists if $j >p$ and $0 \leq \ovr{t-u} \leq j-p$. The vertical map exists and is nonzero for every $w \in I_0^i$,$t \in J$, such that  $\ovr{t-w} \leq p$. The map is well-defined because the composition of horizontal and vertical maps contains $x$ to the power $p$. Maps that have a zero module at one end of the off-diagonal can never be homotopic to zero. The second map is dual, where $w$ lies in $I_0^{i+1}.$

\textbf{(E.3)}
Chain maps with degree  $1$ term zero, double projective. Well-defined for $hr < p$ and every $t \in J$, such that $hr +\ovr{\tt-u} \geq p$, or in the dual, $hr +\ovr{v-\tt} \geq p$.

\[
\begin{array}{cc}
	\xymatrix@C=50pt@R=30pt{
		P_u \oplus P_v \ar[d]\ar@{->}[d]_{\begin{pmatrix} 
			x^{hr}&0\\
			0&0\end{pmatrix} }  \ar[r]^{\begin{pmatrix} x^{\ovr{t-u}}\\ y^{\ovr{v-t}} \end{pmatrix}}   & P_t \ar[d]^{0} \\
		P_u \oplus P_v  \ar[r]^{\begin{pmatrix} x^{\ovr{\tilde{t}-u}} \\ y^{\ovr{v-\tilde{t}}} \end{pmatrix}}   & P_{\tilde{t}} 
	}
\end{array}
\]
\[
\begin{array}{cc}
	\xymatrix@C=50pt@R=30pt{
		P_u \oplus P_v \ar[d]\ar@{->}[d]_{\begin{pmatrix} 
				0&0\\
				0&y^{hr}\end{pmatrix} }  \ar[r]^{\begin{pmatrix} x^{\ovr{t-u}}\\ y^{\ovr{v-t}} \end{pmatrix}}   & P_t \ar[d]^{0} \\
		P_u \oplus P_v  \ar[r]^{\begin{pmatrix} x^{\ovr{\tilde{t}-u}} \\ y^{\ovr{v-\tilde{t}}} \end{pmatrix}}   & P_{\tilde{t}} 
	}
\end{array}
\]

\textbf{(E.4)} Chain maps with degree  $1$ term zero, stalk complex to double  projective.
\[
\begin{array}{cc}
	\xymatrix@C=50pt@R=30pt{
		P_u \ar[d]\ar@{->}[d]_{\begin{pmatrix} 
			x^{hr}&0
			\end{pmatrix} }   \\
		P_u \oplus P_v  \ar[r]^{\begin{pmatrix} x^{\ovr{{t}-u}} \\ y^{\ovr{v-{t}}} \end{pmatrix}}   & P_{{t}} 
	}
\end{array}
\begin{array}{cc}
	\xymatrix@C=50pt@R=30pt{
		P_v \ar[d]\ar@{->}[d]_{\begin{pmatrix} 
				0&y^{hr}\end{pmatrix} }  \\
		P_u \oplus P_v  \ar[r]^{\begin{pmatrix} x^{\ovr{t-u}} \\ y^{\ovr{v-t}} \end{pmatrix}}   & P_{t} 
	}
\end{array}
\]

The first chain map is well-defined if $hr+\ovr{t-u}\geq p$. Dually, the second chain map is
well-defined if $hr+\ovr{v-t}\geq p$. Maps that have a zero module at one end of the off-diagonal can never be homotopic to zero.

\textbf{(E.5)} Maps between stalk complexes allowed by Lemma \ref{homogeneous map lem} must be well-defined and of course cannot be homotopic to zero, since there is no square.

\item \textbf{All maps internal to $I_0^i \cup J_i \cup I_0^{i+1}$ are linear combinations of compositions of the maps in the catalog}
\end{enumerate}
In this section of the proof, we will consider pairs $P'_t$ and $P'_{\tilde{t}}$ of indecomposable components of the tilting complex of the elementary equivalence determined by the subinterval $I_0$. The irreducible components are of two kinds, stalk complexes when $t \in I_0$ and complexes of length two when $t \in I-I_0$.  We will first determine the general well-defined chain map, not necessarily homogeneous, between this pair of complexes, having constant recourse to Lemma \ref{homogeneous map lem}, which gives the possible maps between the indecomposable projectives $P_i$. We then attempt to write the general map as a linear combination of non-zero, well-defined homogeneous chain maps. Finally, we will decompose each homogeneous chain map as a composition of homogeneous chain maps from the catalog. In Case 2.1, we will deal with all maps between stalk complexes, in Cases 2.2, 2.3, and 2.4 we will deal with maps between two-term complexes, and in Case 2.5, with maps between stalk complexes and two-term complexes. We establish a general notation for maps between two-term complexes  $P'_{t}$ and $P'_{\tilde{t}}$ where $t,\tilde{t} \in J$,   $t\preceq \tilde{t}$.

\[
\begin{array}{cc}
	\xymatrix@C=70pt@R=30pt{
		P'_t: Q_{t} \ar[d]^{\ell_0} \ar@{->}[d]  \ar[r]  & P_t \ar[d]^{\ell_1} \\
		P'_{\tilde{t}}: Q_{\tilde{t}} \ar@/^/[u]^{k_0} \ar[r]  & P_{\tilde{t}} \ar@/^/[u]^{k_1}
	}
\end{array}
\]

\noindent where Lemma \ref{homogeneous map lem} gives the maps $\ell_1$ and $k_1$.,
  
\textbf{Case 2.1}:
First, we check the maps between stalk complexes, where $t,\tilde{t} \in I_0$. In general, we assume that if $t$ ,$\tilde{t}$ in from  the same interval then $t  \preceq \tilde{t}$, as in Def.  \ref{sec:Definition in}.  We get the maps $\ell_0$ and $k_0$ rrom Lemma \ref{homogeneous map lem}.
\[
\begin{array}{cc}
	\xymatrix@C=70pt@R=30pt{
		P_{t} \ar[d]^{\ell_0} \\
		P_{\tilde{t}} \ar@/^/[u]^{k_0}\\
	}
\end{array}
\]

Suppose  $t,\tilde{t} \in I_0^i$ or $t, \tilde{t} \in I_0^{i+1}$.   Each individual map $x^{\ovr{\tilde{t}-t}}(xy)^q$ or $y^{\ovr{\tilde{t}-t}}(xy)^q$ is a composition of maps $\gamma_{\bar{t},\bar{t}+1}$ in the catalog, where $\bar{t}=t,t+1,t+2, \dots,\tilde{t}-1$, followed by $q$ pairs of maps $xy$, also compositions of the $\gamma$ or the maps $\epsilon_t$. In the special cases listed in E.5, one of the maps from E.5 might be needed. The general map is then a linear combination of these homogeneous maps with coefficients $a_q$ or $b_q$.

Suppose that $t  \in I_0^i$, $\tilde{t} \in I_0^{i+1}$, and
 $\ovr{\tt-t}<p$, which implies $j<p$,
the general map given above is a linear combination with coefficients $a_q$ or $b_q$ of the homogeneous maps in the sums, each of which is a composition of the $\gamma_{k,k+1}$, sometimes including $\epsilon_t$, with the jump map from Diagram 3, given by $x^j$ or $y^j$,
followed by $q$ self-maps $xy$, each
  of which is a composition of two maps $\gamma_{t, t+1}$ and $\gamma_{t+1, t}$ or a copy of $\epsilon_t$ maps.
 
 \textbf{Case 2.2}: Single projective, where $Q_t$,$Q_{\tilde{t}}=P_u$. We see from the diagrams in the catalog that such $Q_t,Q_{\tilde{t}}$ occur only when $j>p$, where we have $u+1\preceq  t \preceq \tilde{t}\preceq  u+(j-p)$. By  Lemma \ref{homogeneous map lem}, in a special case where $i=k$ and under the assumption $j>p$, all the non-zero maps $l_0, k_0$ between $P_u \to P_u$ have the form $\sum_{q=0}^{p-1} c_q (xy)^{q}$ and $\sum_{w=0}^{p-1} d_w (xy)^{w}$ respectively.

\[
\begin{array}{cc}
	\xymatrix@C=70pt@R=30pt{
P_u\ar[d]\ar@{->}[d]_{\ell_0} \ar[r]^{x^{\ovr{t-u}}} & P_t \ar[d]^{\ell_1} \\
P_u \ar[r]^{x^{\ovr{\tilde{t}-u}}} & P_{\tilde{t}} 
}
\quad
\xymatrix@C=70pt@R=30pt{
P_u\ar[d]\ar@{->}[d]_{k_0} \ar[r]^{x^{\ovr{\tilde{t}-u}}} & P_{\tilde{t}} \ar[d]^{k_1} \\
P_u \ar[r]^{x^{\ovr{t-u}}} & P_{t}
}
\end{array}
\]
Now we try to simplify these sums by subtracting maps that are homotopic to zero.  Considering the left-hand square,   we are in the case given in Lemma \ref{short} with $s=u$, and conclude that  $\overline{t-u}+\overline{\tilde{t}-t}=\overline{\tilde{t}-u}$. If we have a homogeneous map in degree zero of $(xy)^q$, then the chain map  is well-defined if the map in degree $1$ is  
$x^{\overline{\tilde{t}-t}}(xy)^q$.

When $q \geq \overline{t-u}$ then
we then have a homotopy to zero  
$T:P_t \rightarrow P_u$ given by $(y^{\overline{t-u}}(xy)^{q-\overline{t-u}})$. The composition of the upper horizontal map and $T$ in the upper triangle gives $(xy)^q$ as needed, while the composition of $T$ with the lower horizontal maps, using Lemma \ref{short} gives 
$(xy)^{q-\overline{t-u}}y^{\overline{t-u}}x^{\overline{t-u}+
\overline{\tilde{t}-t}}$. Pulling out the maximal power of $xy$ and moving it to the end using commutativity, we have the desired vertical map $x^{\overline{\tilde{t}-t}}(xy)^q$.
Subtracting off these homotopic-to-zero chain maps, we can assume that $\ell_0=\sum_{q=0}^{\min(p-1,\overline{t-u})}c_q(xy)^q$.

For each $q$ in this range, we subtract off $c_q$ times the composition of maps from the catalog $\delta^q \circ \alpha_{t,\tilde{t}}$, and we get the situation that $\ell_0=0$. For such a map to be well-defined, we must have $0=x^{\overline{t-u}}\circ \ell_1$. These maps are compositions of the maps in E.1, composed with $\delta_t^q$.  Similarly for the map $k_1$ or $\delta_{\tt}^w$. The case of $Q_t, Q_{\tt}=P_v$ is dual.

 \textbf{Case 2.3}:  Double projectives, where $Q_t=Q_{\tilde{t}}=P_u \oplus P_v $. Of all the five cases, this is the most difficult, since we may have $r<p$, in which case we get the full sums in Lemma \ref{homogeneous map lem}.

 Looking at the powers of $x$ that occur in the horizontal maps in Diagram 2, where $p \leq j \leq 2(p-1)$ and recalling that $j=\overline{v-u}$, then $\ovr{(j-p)+1}\leq \overline{t-u}\leq \overline{\tt-u} \leq \ovr{p-1}$. According to Diagram 3, where $2 \leq j \leq p-1$ then $1 \leq \ovr{t-u}\leq \ovr {\tt-u} \leq \ovr{j-1}$.

We recall that $u\in I_0^i$, $v \in I_0^{i+1}$. Assume that $t \preceq \tilde{t} \in J$. Let $a \in P_u$ and $b \in P_v$. From our calculation of maps among projectives in Lemma \ref{homogeneous map lem}, we have that all maps  $P'_t \to P'_{\tilde{t}}$  are of the following form

\[
\begin{array}{cc}
\xymatrix@C=50pt@R=30pt{
P_u \oplus P_v \ar[d]\ar@{->}[d]_{\begin{pmatrix} 
a_{11}(x,y)&a_{12}(x,y)\\
a_{21}(x,y)&a_{22}(x,y)\end{pmatrix}}  \ar[r]^{\begin{pmatrix} x^{\ovr{t-u}}\\ y^{\ovr{v-t}} \end{pmatrix}}   & P_t \ar[d]^{\ell_{1}(x,y)} \\
P_u \oplus P_v  \ar[r]^{\begin{pmatrix} x^{\ovr{\tilde{t}-u}} \\ y^{\ovr{v-\tilde{t}}} \end{pmatrix}}   & P_{\tilde{t}} 
}
\end{array}
\]
\noindent where, by Lemma \ref{homogeneous map lem},
\[
a_{ii}(x,y)=
\sum_{h=0}^{\lfloor  \frac{p-1}{r}\rfloor} \sum_{q_{ii}=0}^{ (p-1-hr)}  a_{q_{ii}}^{iih} \cdot x^{hr}(xy)^{q_{ii}}
\]

\[
+\sum_{h=0}^{\lfloor  \frac{p-1}{r}\rfloor} \sum_{q'_{ii}=0}^{( p-1-hr) }  b_{q'_{ii}}^{iih} \cdot y^{hr}(xy)^{q'_{ii}}
\]

\[
a_{12}(x,y)=
1_{\overline{v-u}<p} \sum_{h=0}^{\lfloor  \frac{p-1-\overline{v-u}}{r}\rfloor} \sum_{q_{12}=0}^{ p-1-(\overline{\tt-t}+hr)}  a_{q_{12}}^{12h} \cdot x^{\overline{v-u}+hr}(xy)^{q_{12}}
\]

\[
+1_{\overline{u-v}<p} \sum_{h=0}^{\lfloor  \frac{p-1-\overline{t-\tt}}{r}\rfloor} \sum_{q'_{12}=0}^{( p-1-(\overline{v-u}+hr)) }  b_q^{12h} \cdot y^{\overline{u-v}+hr}(xy)^{q_{21}}
\]

\[
a_{21}(x,y)=
1_{\overline{u-v}<p} \sum_{h=0}^{\lfloor  \frac{p-1-\overline{u-v}}{r}\rfloor} \sum_{q_{21}=0}^{ (p-1-(\overline{u-v}+hr) )}  a_q^{21h} \cdot x^{\overline{u-v}+hr}(xy)^{q'_{21}}
\]

\[
+1_{\overline{v-u}<p} \sum_{h=0}^{\lfloor  \frac{p-1-\overline{u-v}}{r}\rfloor} \sum_{q_{21}=0}^{ (p-1-(\overline{u-v}+hr)) }  b_q^{21h} \cdot y^{\overline{}+hr}(xy)^{q'_{21}}
\]
\noindent and, finally,
\[
\ell_{1}(x,y)=
1_{\overline{\tt-t}<p} \sum_{h=0}^{\lfloor  \frac{p-1-\overline{\tt-t}}{r}\rfloor} \sum_{q=0}^{( p-1-(\overline{\tt-t}+hr) )}c_q^h \cdot x^{\overline{\tt-t}+hr}(xy)^{q}
\]

\[
+1_{\overline{t-\tt}<p} \sum_{h=0}^{\lfloor  \frac{p-1-\overline{t-\tt}}{r}\rfloor} \sum_{q'=\overline{\tt-t}}^{(p-1-(\overline{t-\tt
	}+hr))}  d_{q'}^h \cdot y^{\overline{t-\tt}+hr}(xy)^{q'}
\]

We  now do diagram chasing on our diagram. 
When calculating the diagonal map using the upper triangle,  we get 
\begin{center}
	$(a,b) \cdot \begin{pmatrix} x^{\ovr{t-u}} \\ y^{\ovr{v-t}} \end{pmatrix} \cdot \ell_1(x,y)=
	(ax^{\ovr{t-u}}+by^{\ovr{v-t}} )\cdot \ell_1(x,y)$ 
\end{center}

\noindent When we use the lower triangle, we have 

\begin{center}
	$(a,b) \cdot {\begin{pmatrix} 
			a_{11}(x,y)&a_{12}(x,y)\\
			a_{21}(x,y)&a_{22}(x,y)\end{pmatrix} }\cdot \begin{pmatrix} x^{\ovr{\tilde{t}-u}} \\ y^{\ovr{v-\tilde{t}}} \end{pmatrix}$ 
\end{center}
\begin{align*}
	=(aa_{11}(x,y)+ba_{21}(x,y),aa_{12}(x,y)+ba_{22}(x,y))\cdot \begin{pmatrix}  x^{\ovr{\tilde{t}-u}} \\ y^{\ovr{v-\tilde{t}}} \end{pmatrix} \\
\end{align*}

The coefficients of $a$ and the coefficients of $b$ must match in these expressions, giving us two equations

\begin{align}
	(a_{11}(x,y) x^{\ovr{\tilde{t}-u}}+a_{12}(x,y)y^{\ovr{v-\tilde{t}}})&=
x^{\ovr{t-u}} \cdot \ell_1(x,y)\\
	(a_{21}(x,y) x^{\ovr{\tilde{t}-u}}+
	a_{22}(x,y)y^{\ovr{v-\tilde{t}}})&=y^{\ovr{v-{t}}}\cdot \ell_1(x,y)
\end{align}

We want to show, under the assumption that this general map is well-defined, that the map is a linear combination of compositions of maps from the catalog. We first try to diagonalize the matrix, using maps from the catalog whose right vertical map is zero.
The first sum in $a_{12}$ is eliminated by subtracting copies of B.1 composed with D.1 and C.2.
When these have been removed, we are left with 
\[
1_{\overline{u-v}<p} \sum_{h=0}^{\lfloor  \frac{p-1-\overline{t-\tt}}{r}\rfloor} \sum_{q'_{12}=0}^{( p-1-(\overline{v-u}+hr)) }  b_q^{12h} \cdot y^{\overline{u-v}+hr}(xy)^{q'_{21}}
\]

The case of the double projective occurs only when $j=\ovr{v-u} \leq 2(p-1)$ and for this second sum to be non-zero we need  $r-j=\ovr{u-v}<p$. Furthermore, if $r>p$, then the upper limit of the first summation is zero, so we must have $h=0$.

 Whenever $q'_{12} \geq \ovr{v-u}$ we can remove 
compositions with B.1, and we will assume that this has been done. We now use the map in C.2 with parameters $b_0, c_0, h_0$, where we assume $c_0=0$ and $h_0=b_0$.  Since this means that $h_0=c_0+b_0$, the map is homotopic to zero, and thus its compositions with 
$\alpha_{t,\tt}$  and with powers of D.2 are also homotopic to zero.  

\[
\begin{array}{cc}
	\xymatrix@C=30pt@R=30pt{
		P_u \oplus P_v \ar[d]\ar@{->}[d]_{\begin{pmatrix} 0 &h_0(xy)^{q-(\ovr{v-t})}x^{\ovr{v-u}}y^{hr} \\0& h_0(xy)^q y^{hr} \end{pmatrix} }  \ar[r]^{\begin{pmatrix} x^{\ovr{t-u}} \\ y^{\ovr{v-t}} \end{pmatrix}}   & P_t \ar[d]^{h_0(xy)^q y^{hr}} \\
		P_u \oplus P_v  \ar[r]^{\begin{pmatrix} x^{\ovr{t-u}} \\ y^{\ovr{v-t}} \end{pmatrix}}   & P_{t} 
	}
\end{array}
\]

Using Lemma \ref{short}, $\ovr{u-v}+\ovr{v-u}=r$ and $\ovr{v-u}=\ovr{v-t}+\ovr{t-u}$, giving us that the upper right-hand entry in the matrix becomes $h_0(xy)^{q+\ovr{t-u}}y^{\ovr{v-u}+(h-1)r}$. Composing with $\alpha_{t,\tt}$  will only multiply by a power of $xy$.  Thus for proper choice of $q'_{12}$, we can use this map to remove the remaining terms of $a_{12}(x,y)$.  A dual procedure for eliminating  $a_{21}(x,y)$ will successfully diagonalize the matrix.

Once the matrix is diagonal, equations 4.1 and 4.2 assume a simple form that allows us to solve immediately for $a_{11}$ and $a_{22}$ whenever the map is well-defined. We separate $\ell_1$ into the two possible sums and treat each separately. We start with the first sum, assuming $\overline{\tt-t}<p$ and solve 

\[
\begin{array}{cc}
	\xymatrix@C=50pt@R=30pt{
		P_u \oplus P_v \ar[d]\ar@{->}[d]_{\begin{pmatrix} 
				a_{11}(x,y)&0\\
				0&a_{22}(x,y)\end{pmatrix} }  \ar[r]^{\begin{pmatrix} x^{\ovr{t-u}}\\ y^{\ovr{v-t}} \end{pmatrix}}   & P_t \ar[d]^{ \sum_{h=0}^{\lfloor  \frac{p-1-\overline{\tt-t}}{r}\rfloor} \sum_{q=0}^{( p-1-(\overline{\tt-t}+hr) )}c_q^h \cdot x^{\overline{\tt-t}+hr}(xy)^{q}} \\
		P_u \oplus P_v  \ar[r]^{\begin{pmatrix} x^{\ovr{\tilde{t}-u}} \\ y^{\ovr{v-\tilde{t}}} \end{pmatrix}}   & P_{\tilde{t}} 
	}
\end{array}
\]
Using 4.1 and the fact that $\ovr{\tt-u}=\ovr{\tt-t}+\ovr{t-u}$, we get 
\[
a_{11}(x,y)= \sum_{h=0}^{\lfloor  \frac{p-1-\overline{\tt-t}}{r}\rfloor} \sum_{q=0}^{( p-1-(\overline{\tt-t}+hr) )}c_q^h \cdot x^{hr}(xy)^{q}
\]
Using 4.2 and the fact that $\ovr{v-t}=\ovr{\tt-t}+\ovr{v-\tt}$
we get
\[
a_{22}(x,y)= \sum_{h=0}^{\lfloor  \frac{p-1-\overline{\tt-t}}{r}\rfloor} \sum_{q=0}^{( p-1-(\overline{\tt-t}+hr) )}c_q^h \cdot x^{hr}(xy)^{\ovr{\tt-t}+q}
\]
The final map obtained from the first sum is a linear combination for various values of $h$ of C.2 composed with D.1 $h$ times, followed by a composition with $\alpha_{t,\tt}$.

Now, assuming that $\ovr{v-t}<p$, we must do something similar with 
the second sum. Instead of writing the chain map again as in the case of the first sum, we will go directly to equations 4.1 and 4.2, using the second sum in $\ell_1$

\begin{align}
	(a_{11}(x,y) x^{\ovr{\tilde{t}-u}})&=
	x^{\ovr{t-u}} \cdot \sum_{h=0}^{\lfloor  \frac{p-1-\overline{t-\tt}}{r}\rfloor} \sum_{q'=\overline{\tt-t}}^{(p-1-(\overline{t-\tt
		}+hr))}  d_{q'}^h \cdot y^{\overline{t-\tt}+hr}(xy)^{q'}\\
	(a_{22}(x,y)y^{\ovr{v-\tilde{t}}})&=y^{\ovr{v-{t}}}\sum_{h=0}^{\lfloor  \frac{p-1-\overline{t-\tt}}{r}\rfloor} \sum_{q'=\overline{\tt-t}}^{(p-1-(\overline{t-\tt
		}+hr))}  d_{q'}^h \cdot y^{\overline{t-\tt}+hr}(xy)^{q'}
\end{align}

Since we cannot simply cancel the power of $x$ in equation (4.3), we must make up the missing power of $x$ from the powers of $xy$, which produces an extra copy of $y^r$.  The final values are 

\begin{align}
	a_{11}(x,y) &=
	 \sum_{h=0}^{\lfloor  \frac{p-1-\overline{t-\tt}}{r}\rfloor} \sum_{q'=\overline{\tt-t}}^{(p-1-(\overline{t-\tt
		}+hr))}  d_{q'}^h \cdot y^{(h+1)r}(xy)^{q'-\ovr{\tt-t}}\\
a_{22}(x,y)&=\sum_{h=0}^{\lfloor  \frac{p-1-\overline{t-\tt}}{r}\rfloor} \sum_{q'=\overline{\tt-t}}^{(p-1-(\overline{t-\tt
		}+hr))}  d_{q'}^h \cdot y^{hr}(xy)^{q'}
\end{align}

The final map obtained from the first sum is a linear combination for various values of $h$ of C.2 composed with D.2 $h$ times, followed by a composition with $\alpha_{t,\tt}$.

\textbf{Case 2.4}: In this case we want to combine single and double projectives. Assume first that $Q_t=P_u$, $Q_{\tilde{t}}=P_u \oplus P_v$. For the map $Q_t \rightarrow P_t$ to be non-zero, we need   $j>p$, as we see from Diagram 2. Since $r \geq j$, for all the maps in Lemma \ref{homogeneous map lem} we have $h=0$.  Since $j=\ovr{v-u}$, there are no direct maps from $P_u$ to $P_v$, so
the most general map is 
\[
\begin{array}{cc}
	\xymatrix@C=50pt@R=30pt{
		P_u  \ar[d]\ar@{->}[d]_{\begin{pmatrix} 
				a_{1}(x,y)&0
			\end{pmatrix} }  \ar[r]^{ x^{\ovr{t-u}}}   & P_t \ar[d]^{\ell_{1}(x,y)} \\
		P_u \oplus P_v  \ar[r]^{\begin{pmatrix} x^{\ovr{\tilde{t}-u}} \\ y^{\ovr{v-\tilde{t}}} \end{pmatrix}}   & P_{\tilde{t}} 
	}
\end{array}
\]
\noindent where $a_1=(xy)^q$ is the only possibility compatible with $r>p$, and then, in order for the map to be well-defined, we must have 
$\ell_1=x^{\ovr{\tt-t}}(xy)^q$. We must have $q \leq p-1$ for the map to be non-zero. This is simply a special case of the map 
$\alpha_{t,\tt}$ composed with $q$ copies of C.1 at the beginning or $q$ copies of $C.2$ at the end.
The case where the map goes from $P_v$ to $P_u \otimes P_v$ is entirely dual because again the diagonal map must be zero.

 Assume now $Q_{\tt}=P_u \oplus P_v $, $Q_{{t}}=P_u$, $j>p$. Again all $h=0$ as in Case 3, and there are no direct maps from $P_u$ to $P_v$. However, in order to give the most general well-defined map we must have the coefficient of the generator of $P_v$ vanishing in the upper triangle as it vanishes in the lower triangle, and for this, we need a high power of $xy$ in $\ell_1(x,y)$, so also in 
$a_1(x,y)$.
\[
\begin{array}{cc}
\xymatrix@C=70pt@R=30pt{
P_u \oplus P_v \ar[d] _{\begin{pmatrix}  a_1(x,y)\\0
\end{pmatrix} }  \ar[r]^{\begin{pmatrix} x^{\ovr{\tt-u}} \\ y^{\ovr{v-\tt}} \end{pmatrix}}   & P_{\tt} \ar[d]^{\ell_1(x,y)} \\
P_v  \ar[r]^{x^{\ovr{t-u}}}  & P_{{t}} 
}
\end{array}
\]
\noindent where $a_1(x,y)=\sum_{q=0}^{p-1}a^1_q(xy)^q$ and $\ell_1=\sum_{q_1=0}^{p-1}c_{q_1}^1y^{\ovr{\tilde{t}-t}}(xy)^{q_1}$. Since $\ovr{v-\tt}+\ovr{\tt-t}=\ovr{v-t}$, in order to have the composition of the upper triangle equal to zero, we must have $q_1 \geq p-\ovr{v-t}$.  We have $\ovr{\tt-u}>\ovr{t-u}$ because $t \preceq \tt$. Now to get the coefficient of the generator of $P_u$ to be equal in both triangles, we require
 $a^1x^{\ovr{t-u}}(xy)^q=c^1x^{\ovr{t-u}}(xy)^{\ovr{\tt-t}+q_1}$. The total power of $x$ will be less than $p$ if $q_1 <p-\ovr{\tt-u}$.  In that case, we take  $q=q_1+\ovr{\tt-t}$ and $a_q^1=c_{q_1}^1$. For $q$ and $q_1$ greater than these critical values, the coefficients can be taken arbitrarily. The case with $P_v$ is dual.
 
 To show that this map can be decomposed as a composition of maps in the catalog, let us start with the case $t=t_0=u+(j-p)$ and $\tt=u+(j-p)+1$.  In this case,
 we have the map in Diagram 2 designated by $\mu_u$, where $q_1=0$ and $\ovr{\tt-t}=1$. We can multiply by a constant to get the general $a_1(x,y)$ and we can compose with maps of type C.1 or C.2 to get higher values of $q_1$. To get more general cases, we use compositions with $\alpha_{t_0,t}$  to get $q=\ovr{t_0+1-t}$ while still having $q_1=0$, and then compose with C.1 and C.2. Finally, to get a general value of $\tt$, we compose with $\alpha_{\tt.t_0+1}$, which increases $q$ to $\ovr{\tt-t}$ and then compose with an additional $q_1$ copies of C.1 or C.2.

\textbf{Case 2.5}: It remains to consider the cases from $P_w' \to P_t'$ or $P_t' \to P_w'$ where $w \in I_0$, $t \in J$.
If $Q_t = P_u$, then $j>p$, so we do not have to worry about the sums in Lemma \ref{homogeneous map lem}.

\[
\textbf{(5.a)} 
\begin{array}{cc}
	\xymatrix@C=70pt@R=30pt{
		P_w\ar[d]\ar@{->}[d]^{ax^{\ovr{u-w}}(xy)^q}  \\
		P_u \ar[r]^{x^{\ovr{t-u}}} & P_{t} 
	}
\end{array}
\textbf{(5.b)} 
\begin{array}{cc}
	\xymatrix@C=70pt@R=30pt{
		P_u \ar[d]^{ay^{\ovr{u-w}}(xy)^q}\ar@{->}[r]^{x^{\ovr{t-u}}} & P_{t} \\
		P_w  	
	}
	
\end{array}
\]

The map \textbf{(5.a)} is non-zero if $\ovr{u-w}+q<p$ and well-defined if $\ovr{t-w}+q \geq p$.  There is no possibility of a non-zero power of $x$ reaching $P_v$ because $j=\ovr{v-u}>p$. 
The map \textbf{(5.b)} is non-zero if $\ovr{u-w}+q < p$ and is automatically well-defined. It is not homotopic to zero as long as $q < \ovr{t-u}$.
It is irreducible when $q=0$. There is a dual case where we replace $x$ by $y$ and $u$ by $v$.

\[
\textbf{(5.c)} 
\begin{array}{cc}
	\xymatrix@C=70pt@R=30pt{
		P_w\ar[d]\ar@{->}[d]^{(a_1,a_2)}  \\
		P_u \oplus P_v\ar[r]^{\begin{pmatrix} x^{\ovr{t-u}} \\ y^{\ovr{v-t}} \end{pmatrix}} & P_{t} 
	}
\end{array}
\textbf{(5.d)} 
\xymatrix@C=70pt@R=30pt{
	P_u \oplus P_v  \ar[d] \ar@{->}[d]^{x^{\ovr{u-w}}}  \ar[r]^{\begin{pmatrix} x^{\ovr{t-u}} \\ y^{\ovr{v-t}} \end{pmatrix}} & P_{t} \\
	P_w
}
\]

The map \textbf{5.c} will be non-zero when one of the summands in $a_1$ or $a_2$ will be non-zero, where
\[
a_{1}(x,y)=
\sum_{h=0}^{\lfloor  \frac{p-1}{r}\rfloor} \sum_{q_{1}=0}^{ (p-1-hr)}  a_{q_{}}^{1h} \cdot x^{\ovr{u-w}+hr}(xy)^{q_{}}
\]

\[
+\sum_{h=0}^{\lfloor  \frac{p-1}{r}\rfloor} \sum_{q'_{ii}=0}^{( p-1-hr) }  b_{q'_{ii}}^{iih} \cdot  x^{\ovr{u-w}} y^{hr}(xy)^{q'_{ii}}	
\]

\[
a_{2}(x,y)=
1_{\overline{v-u}<p} \sum_{h=0}^{\lfloor  \frac{p-1-\ovr{v-u}}{r}\rfloor} \sum_{q_{12}=0}^{ p-1-(\overline{\tt-t}+hr)}  a_{q_{12}}^{12h} \cdot x^{\overline{v-u}+\ovr{u-q}+hr}(xy)^{q_{12}}
\]

\[
+1_{\overline{u-v}<p} \sum_{h=0}^{\lfloor  \frac{p-1-\overline{t-\tt}}{r}\rfloor} \sum_{q'_{12}=0}^{( p-1-(\overline{v-u}+hr)) }  b_q^{12h} \cdot y^{\overline{u-v}+\ovr{u-w}+hr}(xy)^{q_{21}}
\]

It is well-defined only when the composition of the vertical and horizontal arrows is zero.

In summary, we have found all possible maps between $P_t'$ and $P_{\tilde{t}}'$ where $u<t,\tilde{t}<v$, and shown how to factor them using intra-arc maps from the catalog.
\end{proof}
\section{An example of intra-arc  maps$:(C_5 \times C_5) \rtimes C_4$}
\label{4_6}

We return to an example we mentioned before in Example \ref{example 554_2}  with $I=\{0, 1, 2,  3\}$. According to what we did in $\S$ \ref{The structure of projective modules} and Lemma \ref{homogeneous map lem}, the composition factors of the projective modules form diamonds. This case gives us an example of the extra maps that can appear when $r<p$.
We choose the projective $P_2$  arbitrarily and make an elementary equivalence with $I_0=\{0,1,3\}$. We used a mutation for our first example because the resulting tilting complex is simpler. This gives a tilting complex $T'=\bigoplus T'_j$ as in \S \ref{sec:ELEMENTARY EQUIVALENCES}
\[
\begin{array}{ccccccccccccccc}
T_0': &  & 0 & \rightarrow & P_{0} & \rightarrow & 0\\
T_1': &  & 0 & \rightarrow & P_{1}  & \rightarrow & 0\\
T_2': &  & 0 & \rightarrow & P_{1} \oplus P_{3} & \rightarrow & P_{2} & \rightarrow & 0\\
T_3': &  & 0 & \rightarrow & P_{3} &\rightarrow & 0
\end{array}
\]
In order to compute the new projective module $P'_2$ in $\End_{D^{b}(A)}(T',T')$, we want to calculate a basis for $\Hom_{D^{b}(A)}(T',T'_{i})$, $i=0,\dots,3$ , by finding all the maps between all the projectives, as already done in the previous Proposition \ref{internal}. Since $r<p$ we have a reducible maps involving $x^4$ from $P_i \to P_i$ and $y^4$ from $P_i \to P_i$ for $\{ i=0 ,1,3 \}$. These maps commute with each other and their product is the socle of the corresponding projective. There are also reducible loops $xy$ obtained by going to a neighbor by $x$ and returning by $y$, or vice versa. Then $(xy)^4$ also gives the socle.

Since $J=\{2\}$, $j=2$ and $p=5$ we are in the case $j \leq p-1$ and the adjacent maps come from Diagram 3 as below, where $u=1$, $v=3$ all the complexes in the middle section contract to a single line. 

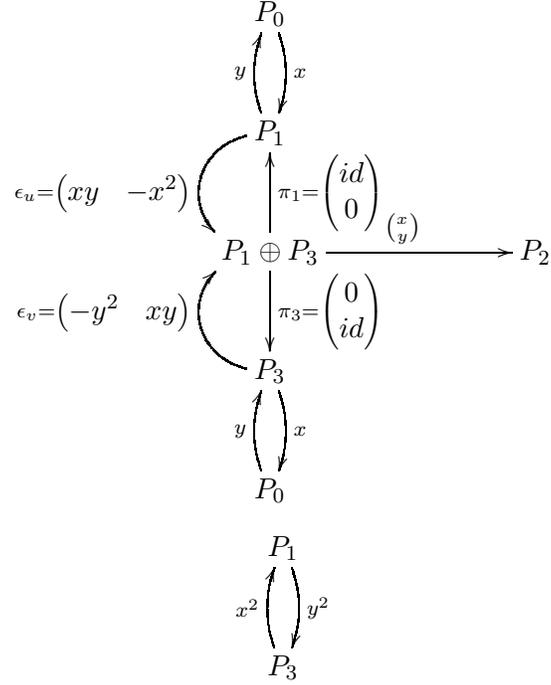
\[\xymatrix@C=70pt@R=30pt{
\\
P_{0} \ar@/^/[d]^x 	\\
P_1  \ar@/^/[u]^y  \ar@(l,l)[d]_{\epsilon_u=\begin{pmatrix}xy &-x^2\end{pmatrix}}	\\
P_1 \oplus P_3 \ar[d]^{\pi_3=\begin{pmatrix} 0\\ id \end{pmatrix} }  \ar[u]_{\pi_1=\begin{pmatrix} id\\ 0 \end{pmatrix} } \ar[r]^{\binom{x^{}}{y}}    &  P_2 	\\
P_3   \ar@(l,l)[u]^{\epsilon_v=\begin{pmatrix} -y^2&xy \end{pmatrix} }\ar@/^/[d]^x \\
P_{0} \ar@/^/[u]^y
}
\]
\[\xymatrix@C=70pt@R=30pt{
P_1  \ar@/^/[d]^{y^2}	\\
P_3  \ar@/^/[u]^{x^2}\\
}
\]	

We compute the quiver, given in Fig. \ref{462_f}.
\begin{figure}
	\[
	\xymatrix@C=100pt@R=30pt{
		& P'_1 \ar@/^/[dddd]^{x^2} \ar@_[ddr] |{\tau_1} \ar@/^/[ddr]|{\epsilon_1}  \ar@/^/[ddl]^y    \\
		&&\\
		P'_0 \ar@/^/[uur]^x \ar@/^/[ddr]^y  & &  P'_2 \ar@(u,ur)^{\iota_1}  \ar@(ur,r)^{\iota_3} \ar@(r,dr)^{\eta_1}  \ar@(dr,d)^{\gamma} \ar@(d,dl)^{\eta_3}
		\ar@/^/[uul]^{\pi_1} \ar@/^/[ddl]^{\pi_3} \\
		&&  \\ 
		& P'_3 \ar@/^/[uuuu]^{y^2}
		\ar@/^/[uur]|{\epsilon_3}  \ar@_[uur] |{\tau_3} \ar@/^/[uul]^x     
	}
	\]
	\caption{}
	\label{462_f}
\end{figure}
We now go through the catalog.  From \textbf{A}, Diagram 3, we have various maps shown between $P'_0, P'_1, P'_3$. For chain maps involving $P'_2$, we let $\ell_0$ be the term in degree zero and $\ell_1$ be the term in degree one. We start with  maps between $P'_2$ and $P'_1,	P'_3$ given by 

\begin{align*}
	\epsilon_1\phantom{^2}  &: P'_1 \to P'_2, \ell_1=0.
	&
	\epsilon_3\phantom{^2}  &: P'_3 \to P'_2, \ell_1=0.
	\\
	&\ell_0=\left(\begin{array}{rr}
		xy, &-x^2 
	\end{array}\right)
	&&\ell_0=\left(\begin{array}{rr}
		-y^2, & xy
	\end{array}\right)
	\\
	\\
	\pi_1\phantom{^2}  &:P'_2 \to P'_1, \ell_1=0.
	&
	\pi_3\phantom{^2}  &:P'_2 \to P'_3, \ell_1=0.
	\\
	&\ell_0=\left(\begin{array}{rr}
		id \\ 0
	\end{array}\right)
	&&\ell_0=\left(\begin{array}{rr}
		0 \\id
	\end{array}\right)
\end{align*}

The self-maps in $B$ are in this case reducible, The map in B.1 is the composition $\epsilon_1\circ \pi_1$ and the map in B.2 is $\epsilon_3\circ \pi_3$. Since the sum is homotopic to zero, we get a relation that $\epsilon_1\circ \pi_1+\epsilon_3\circ \pi_3=0$. Denote by $\gamma$ the map from C.2 with $q=1$.  It satisfies a relation $\gamma^5=0$. It is not reducible because of the contraction of the middle section of Diagram 3 to a single complex $P'_3$, which means that we cannot go back and forth using $x$ and $y$.

 From D.1 and D.2 we get self-maps $\eta_1,\eta_3$ and from D.3 we get two  self-maps $\iota_1, \iota_3$ on $P'_2$, as below
\begin{align*}
	\\
	\eta_1\phantom{^2} &: P'_2 \to P'_2, \ell_1=x^4.	
	&
	\eta_3\phantom{^2} &: P'_2 \to P'_2, \ell_1=y^4.
	\\
	&\ell_0= \left(\begin{array}{rr}
		x^4&0\\
		0&x^4
	\end{array}\right)
	&&\ell_0=\left(\begin{array}{rrrrrr}
		y^4&0\\
		0&y^4 
	\end{array}\right)
	\\
	\iota_1\phantom{^2}  &: P'_2 \to P'_2, \ell_1=x^4.
	&
	\iota_3\phantom{^2} &: P'_2 \to P'_2, \ell_1=y^4.
	\\
	&\ell_0=\left(\begin{array}{rrrrrr}
		0&0\\
		x^3y&0
	\end{array}\right)
	&&\ell_0=\left(\begin{array}{rrrrrr}
		0&xy^3\\
		0&0
	\end{array}\right)
	\\	
\end{align*}
\noindent where  $\eta_1 $ is the difference between the map in D.1 and the first map in E.3, and $\eta_3$ is the difference between  D.2 and the second map in E.3, while  $\iota_1, \iota_3$ are the special cases in D.3. We also get a relation, because $\eta_1 \circ \eta_3 = \eta_3 \circ \eta_1 = \gamma^4$.
\noindent 
In addition, from E.4, we have maps from $P'_i$ to $P'_2$:
\begin{align*}
	\\
	\tau_3\phantom{^2}  &:P'_3 \to P'_2, \ell_1=0.
	&
	\tau_1\phantom{^2}  &:P'_1 \to P'_2, \ell_1=0.
	\\
	&\ell_0= \left(\begin{array}{rr}
		0 &y^4
	\end{array}\right)
	&&\ell_0= \left(\begin{array}{rr}
		x^4 &0
	\end{array}\right)
	\\
\end{align*}

Besides the relations already listed, we can derive  relations from $x^5=0$, $y^5=0$ and $xy=yx$, and in addition, we also get $(\gamma)^5=0$.
We believe that these are not yet all relations.  There is a loop 
$\epsilon_3 \circ y^2 \circ \pi_1$ and its dual, which may combine with some of the self-maps to give a relation.
\renewcommand{\refname}{REFERENCES}


\begin{thebibliography}{999}
\bibitem[C]{C} J. H. Conway, R. T. Curtis,  S. P. Norton, R. A. Parker, R. A. Wilson, The ATLAS of Finite  Groups. Maximal Subgroups and Characters of Simmple Groups, Clarendon Press, Oxford (1985).
\bibitem[CR]{CR} C. Curtis and I.  Reiner, {Representation Theory  of Finite Groups and Associative Algebras}, American Mathematical Society, Chelsea Publishing, (1962).
\bibitem[GS]{GS} M. Gerstenhaber, M.  Schaps, \textit{ The modular version of Maschke's theorem for normal abelian p-Sylows   } Journal of Pure and Applied Algebra 109 (1996),  257-264.
\bibitem[H]{H}D. Happel, \emph{Triangulated Categories in the Representation Theory of Finite Dimensional Algebras}, London Mathematical Society Lecture Note Series, 119. Cambridge University Press, Cambridge, (1988).
\bibitem [K]{K} B. K\"ulshammer, \emph{Crossed products and blocks with normal defect groups}, Comm. Algebra \textbf{13} (1985), no. 1, 147-168. 
\bibitem[KK]{KK} S. Koshitani and N. Kunugi, \emph{Brou\'e conjecture holds for principal $3$-blocks with elementary abelian defect group of order $9$}, Journal of Algebra, \emph{248}, (2002), 575-604.
\bibitem[KZ2]{KZ2}S. König nd A. Zimmermann,\emph{ Derived equivalences for group rings}. With contributions by Bernhard Keller, Markus Linckelmann, Jeremy Rickard and Raphaël Rouquier. Lecture Notes in Mathematics, 1685. Springer-Verlag, Berlin, (1998).
\bibitem [Lan] {Lan} P. Landrock, Finite Group Algebras and their Modules, London Mathematical Society Lecture Note Series 84, Cambridge University Press, Cambridge, (1983), 78-80.
\bibitem[M]{M} J. McKay, \textit{Graphs singularities and finite groups}, Proc. Symp. Pure Math.37 (1980), 183-186.
\bibitem[O]{O} T. Okuyama, \emph{Some examples of derived equivalent blocks of finite groups,} (Japanese)  Proceedings of Conference on Group Theory, (1996):  English translation available as a preprint.
\bibitem[R1]{R1} J. Rickard, \emph{Morita theory for derived categories},  Math. Soc., vol 39, no. 2, (1989), 436--456.
\bibitem[R2]{R2}J. Rickard, Derived equivalence and stable equivalence,  \emph{J. Pure and Appl. Alg.} 61, 303-317, (1989).
\bibitem[RR]{RR} I. Reiten and C. Riedtmanm, \emph{Skew group algebras in the representation theory of artin algebras}, J. Algebra, \textbf{92} (1985), 224-282.
\bibitem[S]{S} M. Schaps, \textit{A modular version of Maschke's theorem for groups with cyclic
p-Sylow subgroup}, J. of Algebra, 163 (1994), 623-635.
\bibitem[SSS]{SSS} M. Schaps, D. Shapira and O. Shlomo, \textit{Quivers of blocks with normal defect group}, American Mathematical Society Proc. Symp. Pure Math.
63 (1998), 497-510.
\bibitem[SZ2]{SZ2}M. Schaps, E. Zakay-Illouz,\emph{Pointed Brauer trees}. J. Algebra 246, no. 2, (2001), 647--672.
\bibitem[SZ]{SZ} M. Schaps and Z. Zvi,  \emph{Mutations and pointing for Brauer trees algebras}, arXiv 1606.04341, (2016), Osaka J. Math. \textbf{57},(2020), 689-709.
\bibitem [Sw]{Sw} W. Schwarz, \emph{Die struktur modularer gruppenringe endlicher gruppen der p-L\"ange 1} , J. Algebra, \textbf{60} (1979), 51-75.
\bibitem[W]{W}C.A. Weibel, {An Introduction to Homological Algebra}, Cambridge Studies in Advanced Mathematics, 38. Cambridge University Press, Cambridge, (1994).
\end{thebibliography}
 \end{document}